\RequirePackage{fix-cm}
\documentclass[smallextended]{svjour3} 
\smartqed 
\usepackage{graphicx}
\usepackage{mathptmx} 
\usepackage{latexsym,amsfonts,amsmath,amssymb}
\renewcommand{\boldsymbol}[1]{\pmb{#1}} 
\newcommand{\satop}[2]{\stackrel{\scriptstyle{#1}}{\scriptstyle{#2}}}
\newcommand{\bsDelta}{{\boldsymbol{\Delta}}}
\newcommand{\bsgamma}{{\boldsymbol{\gamma}}}
\newcommand{\bszero}{{\boldsymbol{0}}}

\newcommand{\bse}{{\boldsymbol{e}}}

\newcommand{\bsx}{{\boldsymbol{x}}}
\newcommand{\bsy}{{\boldsymbol{y}}}
\newcommand{\bsz}{{\boldsymbol{z}}}

\newcommand{\bsnu}{{\boldsymbol{\nu}}}
\newcommand{\matrB}{{\boldsymbol B}}
\newcommand{\rd}{\mathrm{d}}
\newcommand{\bbA}{\mathbb{A}}
\newcommand{\bbB}{\mathbb{B}}
\newcommand{\bbR}{\mathbb{R}}
\newcommand{\bbZ}{\mathbb{Z}}
\newcommand{\bbN}{\mathbb{N}}
\newcommand{\bbE}{\mathbb{E}}

\newcommand{\calO}{\mathcal{O}}
\newcommand{\calI}{\mathcal{I}}

\newcommand{\calP}{\mathcal{P}}

\newcommand{\calW}{\mathcal{W}}
\newcommand{\setu}{\mathrm{\mathfrak{u}}}
\newcommand{\setv}{\mathrm{\mathfrak{v}}}
\newcommand{\setw}{\mathrm{\mathfrak{w}}}

\newcommand{\KL}{Karhunen-Lo\`eve }
\newcommand{\indx}{{\mathfrak F}}
\newcommand{\supp}{{\mathrm{supp}}}
\newcommand{\mask}[1]{{}}
\usepackage[usenames]{color}
\definecolor{darkred}{RGB}{139,0,0}
\definecolor{darkgreen}{RGB}{0,100,0}
\definecolor{darkmagenta}{RGB}{139,0,139}

\def\cI{{\cal I}}

\def\cT{{\cal T}}
\def\indx{{\mathfrak F}}
\newcommand{\be}{\begin{equation}}
\newcommand{\ee}{\end{equation}}
\newcommand{\bea}{\begin{eqnarray}}
\newcommand{\eea}{\end{eqnarray}}
\newcommand{\beas}{\begin{eqnarray*}}
\newcommand{\eeas}{\end{eqnarray*}}

\begin{document}
\title{Multi-level quasi-Monte Carlo Finite Element methods for a class of
elliptic PDEs with random coefficients}
\titlerunning{Multi-level QMC FE Methods for Elliptic PDEs with Random
Coefficients} 
\author{Frances~Y.~Kuo \and Christoph~Schwab \and Ian~H.~Sloan}
\institute{Frances~Y.~Kuo \at
              School of Mathematics and Statistics,
              University of New South Wales, Sydney NSW 2052, Australia \\
              \email{f.kuo@unsw.edu.au}
           \and
           Christoph~Schwab \at
              Seminar for Applied Mathematics, ETH Z\"urich, ETH Zentrum,
              HG G57.1, CH8092 Z\"urich, Switzerland \\
              \email{christoph.schwab@sam.math.ethz.ch}
           \and
           Ian~H.~Sloan \at
              School of Mathematics and Statistics,
              University of New South Wales, Sydney NSW 2052, Australia \\
              \email{i.sloan@unsw.edu.au}
}
\date{
      May 2014}  
\maketitle
\begin{abstract}
Quasi-Monte Carlo (QMC) methods are applied to multi-level Finite Element
(FE) discretizations of elliptic partial differential equations (PDEs)
with a random coefficient. The representation of the random coefficient is
assumed to require a countably infinite number of terms.

The multi-level FE discretizations are combined with families of QMC
methods (specifically, randomly shifted lattice rules) to estimate
expected values of linear functionals of the solution, as in
\cite{GKNSSS,GKNSS11,KSS1} in the single level setting. Here, the expected
value is considered as an infinite-dimensional integral in the parameter
space corresponding to the randomness induced by the random coefficient.
In this paper we study the same model as in \cite{KSS1}. The error
analysis of \cite{KSS1} is generalized to a multi-level scheme, with the
number of QMC points depending on the discretization level, and with a
level-dependent dimension truncation strategy. In some scenarios, it is
shown that the overall error of the expected value of the functionals of
the solution (i.e., the root-mean-square error averaged over all shifts)
is of order $\calO(h^2)$, where $h$ is the finest FE mesh width, or
$\calO(N^{-1+\delta})$ for arbitrary $\delta>0$, where $N$ denotes the
maximal number of QMC sampling points in the parameter space. For these
scenarios, the total work for all PDE solves in the multi-level QMC FE
method is shown to be essentially of the order of {\em one single PDE
solve at the finest FE discretization level}, for spatial dimension $d=2$
with linear elements.

The analysis exploits regularity of the parametric solution with respect
to both the physical variables (the variables in the physical domain) and
the parametric variables (the parameters corresponding to randomness). As
in \cite{KSS1}, families of QMC rules with ``POD weights'' (``product and
order dependent weights") which quantify the relative importance of
subsets of the variables are found to be natural for proving convergence
rates of QMC errors that are independent of the number of parametric
variables. Our POD weights for the multi-level QMC FE algorithm are
different from those for the single level algorithm in \cite{KSS1}.
\end{abstract}
\section{Introduction} \label{sec:intro}
This paper is a sequel to our work \cite{KSS1}, where we analyzed
theoretically the application of quasi-Monte Carlo (QMC) methods combined
with finite element (FE) methods for a scalar, second order elliptic
partial differential equation (PDE) with random diffusion. The diffusion
is assumed to be given as an infinite series with random coefficients. As
in \cite{KSS1}, we consider the model parametric elliptic Dirichlet
problem
\begin{equation}\label{eq:PDE1}
 - \nabla \cdot (a(\bsx,\bsy)\,\nabla u(\bsx,\bsy))
 \,=\, f(\bsx) \quad\mbox{in}\quad D \subset \bbR^d\;,
 \quad
 u(\bsx,\bsy) \,=\, 0 \quad\mbox{on}\quad \partial D\;,
\end{equation}
for $D\subset\bbR^d$ a bounded domain with a Lipschitz boundary $\partial
D$, where $d=1,2$, or $3$ is assumed given and fixed (we do not track the
dependence of constants on $d$ in this work). In \eqref{eq:PDE1}, the
gradients are understood to be with respect to the physical variable
$\bsx$ which belongs to $D$, and the parameter vector $\bsy = (y_j)_{j\geq
1}$ consists of a countable number of parameters $y_j$ which we assume, as
in \cite{KSS1}, to be i.i.d.\ uniformly distributed. Hence, we assume
\[
  \bsy\in [-\tfrac{1}{2},\tfrac{1}{2}]^\bbN \,=:\, U\;.
\]
The parameter $\bsy$ is thus distributed on $U$ with the uniform
probability measure $\mu(\rd\bsy) = \bigotimes_{j\geq 1} \rd y_j =
\rd\bsy$. This simple probability model readily lends itself to treatment
by QMC integration.

The parametric diffusion coefficient $a(\bsx,\bsy)$ in \eqref{eq:PDE1} is
assumed to depend linearly on the parameters $y_j$ as follows:
\begin{equation} \label{eq:defaxy}
  a(\bsx,\bsy)
  \,=\, \bar{a}(\bsx) + \sum_{j\geq 1} y_j\, \psi_j(\bsx)\;,
  \qquad
  \bsx\in D\;, \quad\bsy\in U\;.
\end{equation}
The $\psi_j$ can arise from either the eigensystem of a covariance
operator (see, e.g. \cite{ST06}), or other suitable function systems in
$L^2(D)$. As in \cite{KSS1} we impose a number of assumptions on $\bar{a}$
and $\psi_j$ as well as on the domain~$D$:
%
\begin{description}
\item [(\textbf{A1})]%
We have $\bar{a} \in L^\infty(D)$ and $\sum_{j\geq
1}\|\psi_j\|_{L^\infty(D)}
< \infty$. 
\item [(\textbf{A2})]%
There exist $a_{\max}$ and $a_{\min}$ such that $0 < a_{\min}\le
a(\bsx,\bsy)\le a_{\max}$ for all $\bsx\in D$ and $\bsy\in
U$. 
\item [(\textbf{A3})]%
There exists $p\in (0,1)$ such that $\sum_{j\ge 1} \|\psi_j\|^p_{L^\infty(D)}< \infty$.
\item [(\textbf{A4})]%
With the norm $\|v\|_{W^{1,\infty}(D)} :=
\max\{\|v\|_{L^\infty(D)},\|\nabla v\|_{L^\infty(D)}\}$,
we have
 $\bar{a} \in W^{1,\infty}(D)$ and $\sum_{j\geq
1}\|\psi_j\|_{W^{1,\infty}(D)}
< \infty$.
\item [(\textbf{A5})]%
The sequence $\psi_j$ is ordered so that
$\|\psi_1\|_{L^\infty(D)}\ge \|\psi_2\|_{L^\infty(D)} \ge\cdots$.
\item [(\textbf{A6})]%
The domain $D$ is a convex and bounded polyhedron with plane faces.
\end{description} %
In this paper we impose one additional assumption:
\begin{description}
\item [(\textbf{A7})]%
For $p$ as in (\textbf{A3}), there exists $q \in [p,1]$ such that
$\sum_{j\ge 1} \|\psi_j\|^q_{W^{1,\infty}(D)}< \infty$. 
\end{description}
We now briefly comment on each assumption. Assumption (\textbf{A1})
ensures that the coefficient $a(\bsx,\bsy)$ is well-defined for all
parameters $\bsy\in U$. Assumption (\textbf{A2}) yields the strong
ellipticity needed for the standard FE analysis. Assumption (\textbf{A3})
is stronger than the second part of Assumption (\textbf{A1}). This
assumption implies decay of the fluctuation coefficients $\psi_j$, with
faster decay for smaller $p$. The value of $p$ determines the convergence
rate in the previous paper \cite{KSS1}. Assumption (\textbf{A4})
guarantees that the FE solutions converge to the solution of
\eqref{eq:PDE1}. Assumption (\textbf{A5}) allows the truncation of the
infinite sum in \eqref{eq:defaxy} to, say, $s$ terms. This assumption is
not needed in this paper when the functions $\psi_j$ satisfy an
orthogonality property in relation to the FE spaces, see \S
\ref{ssec:ortho} below. Assumption (\textbf{A6}) only simplifies the FE
analysis and can be substantially relaxed. Finally, Assumption
(\textbf{A7}) is often stronger than Assumptions (\textbf{A3}) and
(\textbf{A4}). The value of $q\in [p,1]$ as well as that of $p\in (0,1)$
will determine the QMC convergence rates to be shown in this paper.

Our aim in this paper is to extend the QMC FE algorithm of \cite{KSS1} for
the efficient computation of expected values of continuous linear
functionals of the solution of \eqref{eq:PDE1} to a \emph{multi-level}
setting so that the overall computational cost is substantially reduced.
Suppose the continuous linear functional is $G: H_0^1(D) \mapsto \bbR$
(later we may impose stronger regularity assumption on $G$, e.g., $G\in
L^2(D)$). We are interested in approximating the integral
\begin{align} \label{eq:integral}
  I(G(u))
  &\,:=\, \int_U G(u(\cdot,\bsy)) \,\rd\bsy
  \,:=\, \lim_{s\to\infty} I_s(G(u))\;,
\end{align}
where
\begin{align*}
  I_s(G(u)) \,:=\,
  \int_{[-\frac{1}{2},\frac{1}{2}]^s}
  G(u(\cdot,(y_1,\ldots,y_s,0,0,\ldots)))\,\rd y_1\cdots\rd y_s\;. \nonumber
\end{align*}
The (single level) strategy in \cite{KSS1} was to (i) truncate the
infinite sum in the expansion of the coefficient to $s$ terms, (ii)
approximate the solution of the truncated PDE problem using a FE method
with mesh width $h$, and (iii) approximate the integral using a QMC method
(an equal-weight quadrature rule) with $N$ points in $s$ dimensions. The
QMC FE algorithm can therefore be expressed as
\[ 
  Q_{s,N}(G(u^s_h))
  \,:=\, \frac{1}{N} \sum_{i=1}^N G\big(u^s_h(\cdot,\bsy^{(i)})\big)\;,
\] 
where $u^s_h$ denotes the FE solution of the truncated PDE problem, and
$\bsy^{(1)}, \ldots, \bsy^{(N)}$ are QMC sample points which are
judiciously chosen from the $s$-dimensional unit cube
$[-\tfrac{1}{2},\tfrac{1}{2}]^s$. More precisely, the QMC rules considered
in \cite{KSS1} are \emph{randomly shifted lattice rules}; more details
will be given in the next section. It was established in \cite{KSS1} that
the root-mean-square of the error $I(G(u)) - Q_{s,N}(G(u^s_h))$ over all
random shifts is a \emph{sum} of three parts: a truncation error, a QMC
error, and a FE error. For example, in the particular case where
Assumption (\textbf{A3}) holds with $p=2/3$ and $f,G\in L^2(D)$, it was
shown that the three additive parts of the error are of orders
$\calO(s^{-1})$, $\calO(N^{-1+\delta})$, and
$\calO(h^2)=\calO(M_h^{-2/d})$, respectively, where $M_h$ is the number of
FE nodes and $d$ is the spatial dimension. Assuming the availability of a
linear complexity FE solver in the domain $D$ (e.g., a multigrid method),
the overall cost of the (single level) QMC FE algorithm is $\calO(s\,N\,
M_h)$. There, as in the present paper, we assume that the functions
$\psi_j$ and their (piecewise-constant) gradients are explicitly known,
and that integration of any FE basis functions over a single element in
the FE mesh is available at unit cost. In effect, we assume that the
entries of the FE stiffness matrix can be computed exactly. The assessment
of the impact of quadrature errors in the FE method is a classical
problem, which is well studied and covered in texts, such as the monograph
of Ciarlet \cite{Ciarlet}.

The purpose of the present paper is the design and the error-versus-cost
analysis of a \emph{multi-level} extension of the single level algorithm
developed in \cite{KSS1}.
The multi-level algorithm takes the form
\begin{equation} \label{eq:ML}
  Q_*^L(G(u))
  \,:=\, \sum_{\ell=0}^L Q_{s_\ell,N_\ell}
  \Big(G \big(u_{h_\ell}^{s_\ell} - u_{h_{\ell-1}}^{s_{\ell-1}}\big)\Big)\;,
\end{equation}
where $\{s_\ell\}_{\ell\ge 0}$ is a nondecreasing sequence of truncation
dimensions, $u_{h_\ell}^{s_\ell}$ denotes the FE approximation with mesh
width $h_\ell$ of the PDE problem with parametric input \eqref{eq:defaxy}
truncated at $s_\ell$ terms, with the convention
$u_{h_{-1}}^{s_{-1}}\equiv 0$, and $Q_{s_\ell,N_\ell}$ denotes the
(randomly shifted) QMC quadrature rule with $N_\ell$ points in $s_\ell$
dimensions. (For the practical form of the quadrature rule, including
randomization, see \eqref{eq:prac} below.) Assuming again the availability
of a linear complexity FE solver in the domain $D$, the overall cost of
this multi-level QMC FE algorithm is therefore $\calO(\sum_{\ell=0}^L
s_\ell\,N_\ell\,M_{h_\ell})$ operations. Again we use randomly shifted
lattice rules, and we show that $s_\ell$, $N_\ell$, and $M_{h_\ell}$ enter
the root-mean-square of the error $I(G(u)) - Q_*^L(G(u))$ over all random
shifts in a \emph{combined additive and multiplicative manner}. Upon
choosing $s_\ell$ and $N_\ell$ in relation to $h_\ell$ appropriately at
each level $\ell$, we arrive at a dramatically reduced overall cost
compared to the single level algorithm.

The general concept of multi-level algorithms was first introduced by
Heinrich \cite{Hei01} and reinvented by Giles \cite{Gil07,Gil08}. Since
then the concept has been applied in many areas including high dimensional
integration, stochastic differential equations, and several types of PDEs
with random coefficients. Most of these works used multi-level Monte Carlo
(MC) algorithms, while few papers considered multi-level QMC algorithms.
The multi-level QMC FE algorithm \eqref{eq:ML} proposed and analyzed here
differs in several core aspects from the abstract multi-level QMC
framework proposed in \cite{Gne11,HMNR11}. It also differs from the
multi-level MC approach which has recently been developed for elliptic
problems with random input data of the general form \eqref{eq:PDE1} in
\cite{BSZ,CST,CGST,SchwabGittelsonActNum11,TSGU}. The model considered
here, as in \cite{KSS1}, is infinite-dimensional. Previous treatments of
infinite-dimensional quadrature include \cite{Gne11,KSWW10b,HMNR11} with
QMC methods, \cite{NHMR10} with MC methods, and \cite{PW11} with Smolyak
(or sparse-grid) quadrature.

There is an important special case where the functions $\psi_j$ satisfy an
orthogonality property in relation to the FE spaces, see
\eqref{eq:orthprop} ahead. In this case there is \emph{no dimension
truncation error at any level}, that is, with $s_\ell$ chosen in an
appropriate way we have $u_{h_\ell}^{s_\ell}= u_{h_\ell}$. Furthermore,
due to the special structure of the expansion of the coefficient
$a(\bsx,\bsy)$, the overall cost is only $\calO(\sum_{\ell=0}^L
N_\ell\,M_{h_\ell}\,\log(M_{h_\ell}))$ operations. To have this
orthogonality property we need \emph{multiresolution} function systems;
examples are given in \S\ref{ssec:ortho}. We emphasize that the
eigenfunction system of the covariance operator does \emph{not} have this
property.

One of the main findings of the present paper is that the error analysis
of the multi-level QMC FE algorithm requires \emph{smoothness of the
parametric solution simultaneously with respect to the spatial variable
$\bsx$ and to the parametric variable~$\bsy$}. Another key point is that
we require decay of stronger norms of the fluctuation coefficients
$\psi_j$, see Assumption (\textbf{A7}). For the multi-level QMC FE
algorithm, the convergence rate will be determined by both the values of
$q$ in \textbf{(A7)} and $p$ in \textbf{(A3)}, rather than just the value
of $p$ as for the single level algorithm in~\cite{KSS1}.
As in most modern analyses of QMC integration in high dimensions, we use
parameters $\gamma_\setu$, known as \emph{weights}, to describe the
relative importance of the subset of the variables with labels in the
finite subset $\setu\subset\mathbb{N}$. (These weights are to be
distinguished from quadrature weights in, e.g., Gaussian quadrature
formulas.) In \cite{KSS1} the weights were chosen to minimize a certain
upper bound on the product of the \emph{worst case error} and the norm in
the function space, yielding a special form of weights called ``POD
weights'', which stand for ``product and order dependent weights":
\begin{equation} \label{eq:PODweight}
  \gamma_\setu \,=\, \Gamma_{|\setu|}\,\prod_{j\in\setu}\gamma_j\;,
\end{equation}
where $|\setu|$ denotes the cardinality (or the ``order") of the set
$\setu$. These weights are then determined by the two sequences: by
$\Gamma_0 = 1$, $\Gamma_1, \Gamma_2, \Gamma_3, \ldots$ and by
$\gamma_1,\gamma_2,\gamma_3,\ldots$. The error bound obtained in the
present paper is more complicated than the result in \cite{KSS1} due to
the multi-level nature of the algorithm, but we follow the same general
principle for choosing weights. It turns out that the ``optimal" weights
(in the sense of minimizing an upper bound on the overall error) for the
multi-level QMC FE algorithm are again POD weights \eqref{eq:PODweight},
but they are different from the POD weights for the single level algorithm
in \cite{KSS1}. In any case, fast CBC construction algorithms for randomly
shifted lattice rules are available for POD weights, see \cite{DKS13} or
\cite{KSS-survey} for recent surveys, as well as
\cite{SKJ02b,K03,D04,NC06a,NC06b,CKN06,DPW08}.

The outline of this paper is as follows. In \S\ref{sec:review} we
introduce the function spaces used for the analysis and summarize those
results from \cite{KSS1} that are needed for this paper. In
\S\ref{sec:main} we prove the main results required for the error analysis
and combine them to obtain an error bound for the multi-level QMC FE
algorithm. Finally in \S\ref{sec:concl} we give conclusions.
\section{Problem Formulation and Summary of Relevant Results}
\label{sec:review}

\subsection{Function Spaces} \label{ssec:spaces}
First we introduce the function spaces from \cite{KSS1} which
will be used in what follows.
Our variational setting of \eqref{eq:PDE1} is based
on the Sobolev space $V=H^1_0(D)$ and its dual space $V^* = H^{-1}(D)$,
with pivot space $L^2(D)$, and with the norm in $V$ given by
\[
  \|v\|_V \,:=\, \|\nabla v\|_{L^2(D)}\;.
\]
We also consider the Hilbert space with additional regularity with respect
to $\bsx$,
\begin{equation} \label{eq:defZ}
  Z^t \,:=\, \{ v\in V: \Delta v \in H^{-1+t}(D)\}\,,\quad 0 \le t\le 1 \;,
\end{equation}
with the norm
\begin{equation} \label{eq:defnormZ}
  \| v \|_{Z^t} \,:=\,
  \left( \| v \|_{L^2(D)}^2 + \| \Delta v \|_{H^{-1+t}(D)}^2\right)^{1/2}\;,
\end{equation}
where, for $-1\leq r \leq 2$, the $H^r(D)$ norm denotes the homogeneous
$H^r(D)$-norm which is defined in terms of the $L^2(D)$ orthonormalized
eigenfunctions $\varphi_\lambda\in V$ and the eigenvalues $\lambda$ in the
corresponding spectrum $\Sigma$ of the Dirichlet Laplacian in $D$ by
\[ 
\| v \|^2_{H^r(D)}
\,:=\,
\sum_{\lambda\in \Sigma} \lambda^r\, | (v,\varphi_\lambda ) |^2
\;.
\] 
Here, and in the following, we denote by $(\cdot,\cdot)$ the bilinear form
corresponding to the $L^2(D)$ innerproduct, extended by continuity to the
duality pairing $H^r(D)\times H^{-r}(D)$. Standard elliptic regularity
theory (see, e.g.~\cite{GilbargTrudinger}) yields the inclusion
$Z^t\subset H^{1+t}_{\mathrm{loc}}(D)$, and for convex domains $D$ and for
$t=1$ we have $Z^1=H^2(D)\cap H^1_0(D)$. As already seen in
\S\ref{sec:intro}, we will also make use of the norm
\[ 
  \|v\|_{W^{1,\infty}(D)} \,:=\,
  \max\{ \| v \|_{L^\infty(D)} , \|\nabla v\|_{L^\infty(D)} \}\;.
\] 

The integrand in \eqref{eq:integral} is $G(u(\cdot,\bsy))$. To analyze QMC
integration for such integrands, we shall need a function space defined
with respect to $\bsy$. Since our multi-level QMC FE algorithm makes use
of the FE solution $u^s_h$ of the truncated PDE problem to $s$ terms, we
consider the \emph{weighted} and \emph{unanchored} Sobolev space
$\calW_{s,\bsgamma}$, which is a Hilbert space containing functions
defined over the $s$-dimensional unit cube $[-\frac{1}{2},\frac{1}{2}]^s$,
with square integrable mixed first derivatives. More precisely, the norm
for $F= G(u^s_h) \in \calW_{s,\bsgamma}$ is given by
\begin{equation} \label{eq:norm}
  \|F\|_{\calW_{s,\bsgamma}}
  \,:=\,
  \left(
  \sum_{\setu\subseteq\{1:s\}}
  \frac{1}{\gamma_\setu}
  \int_{[-\frac{1}{2},\frac{1}{2}]^{|\setu|}}
  \left|
  \int_{[-\frac{1}{2},\frac{1}{2}]^{s-|\setu|}}
  \frac{\partial^{|\setu|}F}{\partial \bsy_\setu}(\bsy_\setu;\bsy_{-\setu})
  \,\rd\bsy_{-\setu}
  \right|^2 \,
  \rd\bsy_\setu
  \right)^{1/2}\;,
\end{equation}
where $\{1:s\}$ is a shorthand notation for the set $\{1,\ldots,s\}$,
$\frac{\partial^{|\setu|}F}{\partial \bsy_\setu}$ denotes the mixed first
derivative with respect to the ``active'' variables $\bsy_\setu =
(y_j)_{j\in\setu}$, and where $\bsy_{-\setu} =
(y_j)_{j\in\{1:s\}\setminus\setu}$ denotes the ``inactive'' variables.
The ``outer'' integration in \eqref{eq:norm} is omitted when $\setu =
\emptyset$, while the ``inner'' integration is omitted when $\setu =
\{1:s\}$.

Weighted spaces were first introduced by Sloan and Wo\'zniakowski in
\cite{SW98}, and by now there are many variants, see e.g.\
\cite{DSWW04,SWW04}. As in \cite{KSS1}, we have taken the cube to be
centered at the origin (rather than the standard unit cube $[0,1]^s$).
Moreover, we have adopted ``general weights'': there is a weight parameter
$\gamma_\setu$ associated with each group of variables $\bsy_\setu =
(y_j)_{j\in\setu}$ with indices belonging to the set $\setu$, with the
convention that $\gamma_\emptyset = 1$. Later we will focus on ``POD
weights", see \eqref{eq:PODweight}. As in \cite{KSS1}, these POD weights
arise naturally from our analysis for the PDE application.

\subsection{Parametric Weak Formulation} \label{ssec:weak}

As in \cite{KSS1}, we consider the following {\em parameter-dependent
weak formulation of the parametric deterministic problem} \eqref{eq:PDE1}:
for $f\in V^*$ and $\bsy\in U$, find
\begin{equation}\label{eq:paramweakprob}
  u(\cdot,\bsy)\in V:\quad b(\bsy;u(\cdot,\bsy),v)
  \,=\, (f,v) \qquad
  \forall v\in V\;,
\end{equation}
where the parametric bilinear form $b(\bsy;w,v)$ is given by
\[ 
  b(\bsy;w,v)
  \,:=\, \int_D a(\bsx,\bsy)\, \nabla w(\bsx) \cdot \nabla v(\bsx) \,\rd\bsx\;,
  \qquad
 \forall w,v\in V\;.
\] 
It follows from Assumption (\textbf{A2}) that the bilinear form is
continuous and coercive on $V\times V$, and we may infer from the
Lax-Milgram Lemma the existence of a unique solution to
\eqref{eq:paramweakprob} satisfying the standard apriori estimate.
Moreover, additional regularity of the solution with respect to $\bsx$ can
be obtained under additional regularity assumptions on $f$ and the
coefficients $a(\cdot,\bsy)$.
\begin{theorem}[{\cite[Theorems 3.1 and 4.1]{KSS1}}] \label{thm:weak}
Under Assumptions \textnormal{(\textbf{A1})} and
\textnormal{(\textbf{A2})}, for every $f\in V^*$ and every $\bsy\in U$,
there exists a unique solution $u(\cdot,\bsy)\in V$ of the parametric weak
problem \eqref{eq:paramweakprob}, which satisfies
\begin{equation}\label{eq:apriori_V}
  \| u(\cdot,\bsy) \|_V \,\le\, \frac{\|f\|_{V^*}}{a_{\min}} \;.
\end{equation}
If, in addition, $f\in H^{-1+t}(D)$ for some $0\le t\le 1$, and if
Assumption \textnormal{(\textbf{A4})} holds, then there exists a constant
$C>0$ such that for every $\bsy\in U$,
\begin{equation}\label{eq:apriori_Z}
  \| u(\cdot,\bsy) \|_{Z^t} \,\le\, C\, \|f\|_{H^{-1+t}(D)} \;,
\end{equation}
with the norm in $Z^t$ defined by \eqref{eq:defnormZ}.
\end{theorem}

\subsection{Dimension Truncation} \label{ssec:dim-trunc}

Next we summarize a result from \cite{KSS1} needed for estimating the
dimension truncation error. Given $s\in\bbN$ and $\bsy\in U$, we observe
that truncating the sum in \eqref{eq:defaxy} at $s$ terms is the same as
anchoring or setting $y_j=0$ for $j>s$. We denote by $u^s(\bsx,\bsy) :=
u(\bsx,(\bsy_{\{1:s\}};\bszero))$ the solution of the parametric weak
problem \eqref{eq:paramweakprob} corresponding to the parametric diffusion
coefficient \eqref{eq:defaxy} when the sum is truncated after $s$ terms.
As observed in \cite{KSS1}, it will be convenient for the regularity
analysis of \eqref{eq:PDE1} and for the QMC error analysis to introduce
\begin{equation} \label{eq:defbj}
 b_j \,:=\, \frac{\| \psi_j \|_{L^\infty(D)}}{a_{\min}} \;,\qquad
 j\ge 1\;.
\end{equation}
\begin{theorem}[{\cite[Theorem 5.1]{KSS1}}] \label{thm:trunc}
Under Assumptions \textnormal{(\textbf{A1})} and
\textnormal{(\textbf{A2})}, for every $f\in V^*$, every $G\in V^*$, every
$\bsy\in U$ and every $s\in\bbN$, the solution
$u^s(\cdot,\bsy)=u(\cdot,(\bsy_{\{1:s\}};0))$ of the truncated parametric
weak problem \eqref{eq:paramweakprob} satisfies, with $b_j$ as defined in
\eqref{eq:defbj},
\[ 
  \| u(\cdot,\bsy) - u^s(\cdot,\bsy) \|_V
  \,\le\, C\,\frac{\|f\|_{V^*}}{a_{\min}}\sum_{j\ge s+1} b_j
\] 
and
\begin{equation}\label{eq:Idimtrunc}
  |I(G(u))-I_s(G(u))|
  \,\le\, \tilde{C}\,\frac{\|f\|_{V^*}\|G\|_{V^*}}{a_{\min}}
  \bigg(\sum_{j\ge s+1}b_j\bigg)^2
\end{equation}
for some constants $C,\tilde{C}>0$ independent of $s$, $f$ and $G$.
In addition, if Assumptions~\textnormal{(\textbf{A3})} and
\textnormal{(\textbf{A5})} hold, then
\begin{align} \label{eq:stechkin}
  \sum_{j\ge s+1} b_j
  \,\le\,
  \min\left(\frac{1}{1/p-1},1\right)
  \bigg(\sum_{j\ge1} b_j^p \bigg)^{1/p}
  s^{-(1/p-1)}\,.
\end{align}
\end{theorem}

\subsection{Finite Element Discretization} \label{ssec:FE}
Let us denote by $\{ V_h \}_h$ a one-parameter family of subspaces
$V_h\subset V$ of dimensions $M_h < \infty$. Under Assumption
(\textbf{A6}), we think of the spaces $V_h$ as spaces of continuous,
piecewise-linear finite elements on a sequence of regular, simplicial
meshes $\cT_h$ in $D$ obtained from an initial, regular triangulation
$\cT_0$ of $D$ by recursive, uniform bisection of simplices. Then it is
well known (see, e.g., \cite{Ciarlet}) that there exists a constant $C>0$
such that, as $h\to 0$, with the norm in $Z^t$ defined by
\eqref{eq:defnormZ},
\[
  \inf_{v_h \in V_h} \| v - v_h \|_V
  \,\le\, C\, h^t \,\| v \|_{Z^t}\quad
  \mbox{ for all } \; v\in Z^t \;,\quad 0\le t\le 1 \;.
\]
For any $\bsy\in U$, we define the {\em parametric FE approximation}
$u_h(\cdot,\bsy)$ as the FE solution of the parametric deterministic
problem: for $f\in V^*$ and $\bsy\in U$, find
\[
  u_h(\cdot,\bsy) \in V_h: \quad b(\bsy;u_h(\cdot,\bsy),v_h) \,=\,
  ( f, v_h ) \qquad \forall v_h \in V_h\;.
\]
Below we summarize the results from \cite{KSS1} regarding the FE error. We
remark that, by considering the error in approximating a bounded linear
functional, $\calO(h^2)$ convergence for $f, G\in L^2(D)$ follows from an
Aubin-Nitsche duality argument.
\begin{theorem}[{\cite[Theorems 7.1 and 7.2]{KSS1}}] \label{thm:FE}
Under Assumptions \textnormal{(\textbf{A1})}, \textnormal{(\textbf{A2})},
\textnormal{(\textbf{A4})}, and \textnormal{(\textbf{A6})}, for every
$f\in V^*$ and every $\bsy\in U$, the FE approximations $u_h(\cdot,\bsy)$
are stable in the sense that
\[
  \| u_h(\cdot,\bsy) \|_V \,\le\,
  \frac{\| f \|_{V^*}}{a_{\min}} \;.
\]
Moreover, for every $f\in H^{-1+t}(D)$ with $0\le t\le 1$, every $G\in
H^{-1+t'}$ with $0 \le t'\le 1$, and for every $\bsy\in U$, there hold the
asymptotic convergence estimates as $h\rightarrow 0$
\begin{equation} \label{eq:FE1}
  \| u(\cdot,\bsy) - u_h(\cdot,\bsy) \|_V
  \,\le\, C\, h^t\, \|u(\cdot,\bsy) \|_{Z^t}
  \,\le\, C\, h^t\, \|f\|_{H^{-1+t}(D)}
\end{equation}
and
\begin{equation} \label{eq:FE2}
  \left| G(u(\cdot,\bsy)) - G(u_h(\cdot,\bsy)) \right|
  \,\le\, \tilde{C}\, h^{\tau}\, \| f \|_{H^{-1+t}(D)}\, \| G \|_{H^{-1+t'}(D)}\;,
\end{equation}
where $0\le \tau := t + t'\le 2$, and where $C,\tilde{C}>0$ are
independent of $h$ and~$\bsy$.
\end{theorem}

\subsection{QMC Approximation} \label{ssec:QMC}

As in \cite{KSS1}, in this paper we will focus on a family of QMC rules
known as \emph{randomly shifted lattice rules}. For an integral over the
$s$-dimensional unit cube $[-\frac{1}{2},\frac{1}{2}]^s$,
\[
  I_s(F) \,:=\, \int_{[-\frac{1}{2},\frac{1}{2}]^s} F(\bsy)\,\rd\bsy\;,
\]
a realization of an $N$-point randomly shifted lattice rule takes the form
\[
 Q_{s,N}(\bsDelta;F) \,:=\, \frac{1}{N} \sum_{i=1}^N F
  \left(\mathrm{frac}\left(\frac{i\bsz}{N} + \bsDelta\right)
  - \left(\tfrac{1}{2},\ldots,\tfrac{1}{2}\right)
  \right)\;,
\]
where $\bsz\in\bbZ^s$ is known as the \emph{generating vector}, which is
deterministic, while $\bsDelta$ is the \emph{random shift} to be drawn
from the uniform distribution on $[0,1]^s$, and $\mathrm{frac}(\cdot)$
means to take the fractional part of each component in the vector. The
subtraction by the vector $(\frac{1}{2},\ldots,\frac{1}{2})$ describes the
translation from the usual unit cube $[0,1]^s$ to
$[-\frac{1}{2},\frac{1}{2}]^s$. For the weighted Sobolev space
$\calW_{s,\bsgamma}$ with POD weights, good generating vectors $\bsz$ can
be constructed, using a \emph{component-by-component algorithm} at the
cost of $\calO(s\,N\,\log N + s^2N)$ operations, such that the ``shift
averaged" \emph{worst case error} achieves a dimension-independent
convergence rate close to $\calO(N^{-1})$. Moreover, the implied constant
in the big-$\calO$ bound can be independent of $s$ under appropriate
conditions on the weights $\gamma_\setu$. A short summary of these
results, together with references, can be found in \cite[Section~2]{KSS1}.
More detailed surveys can be found in \cite{DKS13} or \cite{KSS-survey}.
For the purpose of this paper, we only need the following bound on the
root-mean-square error.

\begin{theorem}[{\cite[Theorem 2.1]{KSS1}}] \label{thm:QMC}
Let $s,N\in \bbN$ be given, and assume $F\in \calW_{s,\bsgamma}$ for a
particular choice of weights $\bsgamma = (\gamma_\setu)$. Then a randomly
shifted lattice rule can be constructed using a component-by-component
algorithm such that the root-mean-square error satisfies, for all
$\lambda\in (1/2,1]$,
\[
  \sqrt{\bbE\,\left[ |I_s(F) - Q_{s,N}(\cdot;F)|^2 \right]}
  \,\le\,
  \left(
  \sum_{\emptyset\ne\setu\subseteq\{1:s\}} \gamma_\setu^\lambda\,
  [\rho(\lambda)]^{|\setu|}
  \right)^{1/(2\lambda)}
  [\varphi(N)]^{-1/(2\lambda)}\,\|F\|_{\calW_{s,\bsgamma}}\;,
\]
where $\bbE[\cdot]$ denotes the expectation with respect to the random
shift which is uniformly distributed over $[0,1]^s$, $\varphi(N) =
|\{1\le z\le N-1: \gcd(z,N)=1\}|$ denotes the Euler totient function,
\begin{equation} \label{eq:rho}
  \rho(\lambda) \,:=\,
  \frac{2\zeta(2\lambda)}{(2\pi^2)^\lambda}\;,
\end{equation}
and $\zeta(x) = \sum_{k=1}^\infty k^{-x}$ denotes the Riemann zeta
function.
\end{theorem}

For example, when $N$ is prime, $\varphi(N) = N-1$ and a rate of
convergence arbitrarily close to $\calO(N^{-1})$ comes from taking
$\lambda$ in the theorem close to $1/2$. However, note that $\rho(\lambda)
\to\infty$ as $\lambda\to (1/2)+$, making the convergence of the sum over
$\setu$ more and more problematic as $\lambda$ comes closer to $1/2$. For
that reason we shall leave $\lambda$ as a free parameter in the subsequent
discussion.

\section{Multi-level QMC FE Algorithm} \label{sec:main}

\subsection{Formulation of the Multi-level QMC FE Algorithm} \label{ssec:algo}

We are now ready to formulate our multi-level QMC FE algorithm for
approximating the integral \eqref{eq:integral}. Let
\[ 
  h_\ell \,=\, 2^{-\ell}\,h_0 \qquad\mbox{for}\qquad\ell=0,1,2,\ldots \;.
\] 
We suppose that we are given a nested sequence $\{V_{h_\ell}\}_{\ell\ge
0}$ of finite-dimensional subspaces of $V$ of increasing dimension,
\[
  M_{h_0} < M_{h_1} < \cdots < M_{h_\ell} := \dim(V_{h_\ell})
  \,\asymp 2^{d\ell}\, \to \infty \quad\mbox{as}\quad \ell\to\infty
  \;,
\]
where $a_n\asymp b_n$ means there exist $c_1,c_2>0$ such that $c_1
b_n\le a_n\le c_2 b_n$.
In the multi-level method we specify a maximum
level $L$, and with each level~$\ell=0,\ldots,L$ of (uniform) mesh
refinement $\cT_{h_\ell}$ we associate a randomly shifted lattice rule
$Q_{s_\ell,N_\ell}$ which uses $N_\ell$ points in $s_\ell$ dimensions. We
assume moreover that the sequence $\{ s_\ell \}_{\ell=0,\ldots,L}$ of
active dimensions is nondecreasing, i.e.,
\begin{equation}\label{eq:sellincr}
  s_0 \le s_1 \le \cdots \le s_\ell \le s_L\;,
\end{equation}
which implies that the corresponding sets of active coordinates are
nested. To simplify the ensuing presentation, we write
(with slight abuse of notation)
\[ 
  V_\ell \equiv V_{h_\ell}\;,\quad
  \cT_\ell \equiv \cT_{h_\ell}\;,\quad
  Q_\ell \equiv Q_{s_\ell,N_\ell}\;,\quad
  I_\ell \equiv I_{s_\ell}\;,\quad
  u_\ell \equiv u_{h_\ell}^{s_\ell}\;,\quad
  M_\ell \equiv M_{h_\ell} \;.
\] 
Here by $u_{h_\ell}^{s_\ell}$ we mean the FE solution of the truncated
problem with $s_\ell$ terms in the expansion, which is the same as
$u_{h_\ell}(\bsy_{\{1:s_\ell\}};0)$. For convenience we define $u_{-1}
:=0$. Each lattice rule $Q_\ell$ depends on a deterministic generating
vector $\bsz_\ell\in \bbZ^{s_\ell}$, but we shall suppress this dependence
in our notation. A realization of the lattice rule $Q_\ell$ for a draw of
the shift $\bsDelta_\ell\in [0,1]^{s_\ell}$ applied to a function $F$ will
be denoted by $Q_\ell(\bsDelta_\ell;F)$. The random shifts
$\bsDelta_0,\ldots,\bsDelta_L$ are drawn independently from the uniform
distribution on unit cubes of the appropriate dimension. With these
notations, a single realization of our multi-level QMC FE approximation of
$I(G(u))$ is given by
\begin{equation} \label{eq:MLQMCFE}
  Q_*^L(\bsDelta_*;G(u))
  \,:=\, \sum_{\ell=0}^L Q_\ell (\bsDelta_\ell;G(u_\ell-u_{\ell-1}))\;,
\end{equation}
where $\bsDelta_* := (\bsDelta_0,\ldots,\bsDelta_L)$ will be referred to
as the ``compound shift": it comprises all $s_*:=\sum_{\ell=0}^L s_\ell$
components of the random shifts $\bsDelta_\ell$. Equivalently,
$\bsDelta_*$ is drawn from the uniform distribution over $[0,1]^{s_*}$.

The randomly shifted version of \eqref{eq:MLQMCFE} that we use in practice
makes use of $m_\ell$ i.i.d. realizations of the level-$\ell$ shift
$\bsDelta_\ell$, thus takes the form
\begin{equation}\label{eq:prac}
  Q^L(G(u)) \,:=\,
  \sum_{\ell=0}^L \frac{1}{m_\ell}\sum_{i=1}^{m_\ell}
  Q_\ell (\bsDelta_{\ell}^{(i)};G(u_\ell-u_{\ell-1}))\;.
\end{equation}
In the subsequent analysis we work with exact expectations of
\eqref{eq:MLQMCFE}, but in the final section we return to \eqref{eq:prac},
and there justify choosing $m_\ell$ to be a fixed number independent of
$\ell$.

\subsection{Error Analysis of the Multi-level QMC FE Algorithm}
\label{ssec:err1}

Using linearity of $I$, $I_\ell$, $Q_\ell$ and $G$, we can express the
error as
\[
  I(G(u)) - Q_*^L(\bsDelta_*;G(u))
  \,=\, I(G(u)) - \sum_{\ell=0}^L Q_\ell(\bsDelta_\ell;G(u_\ell-u_{\ell-1}))
  \,=\, T_1 + T_2(\bsDelta_*)\;,
\]
where
\begin{align}
  T_1
  &\,:=\, I(G(u)) - \sum_{\ell=0}^L I_\ell (G(u_\ell-u_{\ell-1}))\;, \label{eq:T1} \\
  T_2(\bsDelta_*)
  &\,:=\,
  \sum_{\ell=0}^L (I_\ell- Q_\ell(\bsDelta_\ell)) (G(u_\ell-u_{\ell-1}))\;, \nonumber
\end{align}
where we introduced the operator notation $Q(\Delta)(F):=Q(\Delta;F)$.
Since a randomly shifted lattice rule is an unbiased estimator of the
original integral, it follows that the mean-square error for our
multi-level QMC FE method, i.e., the expectation of the square error with
respect to $\bsDelta_*\in [0,1]^{s_*}$, simplifies to
\begin{align} \label{eq:error-bound}
  \bbE[|I(G(u)) - Q_*^L( \cdot;G(u))|^2]
  &\,=\, T_1^2 + \bbE [T_2^2]\;,
\end{align}
where the cross term vanishes due to $\bbE [T_2] = 0$,
and we have
\begin{align} \label{eq:T2}
  \bbE [T_2^2]
  &\,=\, \sum_{\ell=0}^L \bbE [|(I_\ell-Q_\ell(\cdot)) (G(u_\ell-u_{\ell-1}))|^2]\;,
\end{align}
where the expectation inside the sum over index $\ell$ is with respect to
the random shift $\bsDelta_\ell\in [0,1]^{s_\ell}$.

First we estimate $T_1$ given by \eqref{eq:T1}. Since $u_\ell -
u_{\ell-1}$ only depends on the first $s_\ell$ dimensions, we can replace
$I_\ell(G(u_\ell - u_{\ell-1}))$ by $I(G(u_\ell - u_{\ell-1}))$, and hence
the expression \eqref{eq:T1} simplifies to
\begin{align*}
  T_1
  \,=\,I(G(u-u_L))
  \,=\, I(G(u-u_{h_L})) + I(G(u_{h_L}-u_{h_L}^{s_L}))\;.
\end{align*}
Here $u_{h_L}-u_{h_L}^{s_L}$ is the error that we incur in the FE
approximation by omitting in the coefficient expansion \eqref{eq:defaxy}
all terms with indices $j>s_L$. As we will show in
Theorem~\ref{thm:orthprop} below, {\em this dimension truncation error
vanishes for certain types of \textnormal{(}multiresolution\textnormal{)}
coefficient expansion \eqref{eq:defaxy}}. To allow for this, we introduce
a parameter $\theta_L \in \{0,1\}$, with $\theta_L=1$ in general and
$\theta_L = 0$ indicating that there is no truncation error, and arrive at
the estimate
\begin{align} \label{eq:T1-bound}
  |T_1|
  &\,\le\,
  \sup_{\bsy\in U} |G(u(\cdot,\bsy)-u_{h_L}(\cdot,\bsy))|
  \,+\, \theta_L\, |I(G(u_{h_L}-u_{h_L}^{s_L}))| \nonumber
  \\
  &\,\le\,
  C\,h_L^{\tau}\, \|f\|_{H^{-1+t}(D)}\,\|G\|_{H^{-1+t'}(D)}
  \,+\,
  \theta_L\, \tilde{C}\,
  \frac{\|f\|_{V^*}\,\|G\|_{V^*}}{a_{\min}}
  \bigg(\sum_{j\ge s_L+1} b_j \bigg)^2\;,
\end{align}
where for the first term we applied \eqref{eq:FE2} from
Theorem~\ref{thm:FE}, and for the second term we used \eqref{eq:Idimtrunc}
from Theorem~\ref{thm:trunc} but adapted to the FE solution $u_{h_L}$
instead of $u$.

Next we estimate $\bbE[T_2^2]$ given by \eqref{eq:T2}. We have from
Theorem~\ref{thm:QMC} that
\begin{align} \label{eq:T2-bound}
  \bbE[T_2^2]
  &\,\le\, \sum_{\ell=0}^L
  \left(\sum_{\emptyset\ne\setu\subseteq{\{1:s_\ell\}}}
  \gamma_\setu^\lambda\, [\rho(\lambda)]^{|\setu|}\right)^{1/\lambda}
  [\varphi(N_\ell)]^{-1/\lambda}\,
  \|G(u_{h_\ell}^{s_\ell} - u_{h_{\ell-1}}^{s_{\ell-1}})\|_{\calW_{s_\ell,\bsgamma}}^2
  \;.
\end{align}
To estimate each term in \eqref{eq:T2-bound} for $\ell\ne 0$,
we write
\begin{equation} \label{eq:key}
  \|G(u_{h_\ell}^{s_\ell} - u_{h_{\ell-1}}^{s_{\ell-1}})\|_{\calW_{s_\ell,\bsgamma}}
  \,\le\,
  \|G(u_{h_\ell}^{s_\ell} - u_{h_{\ell-1}}^{s_\ell})\|_{\calW_{s_\ell,\bsgamma}}
  + \|G(u_{h_{\ell-1}}^{s_\ell} - u_{h_{\ell-1}}^{s_{\ell-1}})\|_{\calW_{s_\ell,\bsgamma}}\;.
\end{equation}
In \S\ref{ssec:key} ahead, we bound the two terms in \eqref{eq:key}
separately, and then return to complete the error analysis in
\S\ref{ssec:err2}. Note that the second term in \eqref{eq:key} vanishes if
$s_\ell = s_{\ell-1}$. It also vanishes in the special case when, for all
$\ell\ge 1$ and an appropriately chosen increasing sequence ${s_\ell}$, we
have
$u_{h_{\ell-1}}^{s_{\ell-1}}=u_{h_{\ell-1}}^{s_{\ell}}=u_{h_{\ell-1}}$.
This can happen when there is a special orthogonality property between the
functions $\psi_j$ in the representation~\eqref{eq:defaxy} and the FE
spaces~$V_\ell$. We discuss this very important special case in the next
subsection.

\subsection{A Special Case with an Orthogonality Property}
\label{ssec:ortho}
In this subsection we suppose that the sequence $\psi_j$ has
properties usually associated with a multiresolution analysis of
$L^2(D)$, as shown in the Haar wavelet example below. For this purpose it
is useful to relabel the basis set with a double index, as
\begin{equation} \label{eq:relabel}
  \{\psi_j: j\ge 1\}
  \,=\,
  \{\psi_m^n: n\ge 0,\; m\in J_n\}\;,
\end{equation}
where the first index $n$ indicates the (multiresolution) level, and the
second index $m\in J_n$ indicates the location of a level-$n$ basis
function within $D$, with $J_n$ denoting the set of all location indices
at level $n$. We suppose that all basis functions $\psi_m^n$ at level~$n$
are piecewise polynomial functions on the triangulation $\cT_n$, and have
isotropic support whose diameter is of exact order $h_n$,
implying $|J_n|\asymp 2^{dn}$.
\begin{definition} \label{def:ortho}
Let $S^0(D,\cT)$ and $S^1(D,\cT)$ be the subspaces defined by
\begin{align*}
S^0(D,\cT)&\,:=\, \{v\in L^2(D) \;:\; v|_{K}\in P^0(K) \mbox{ for all } K\in \cT\}\;,
\\
S^1(D,\cT) &\,:=\, \{v\in H^1_0(D) \;:\; v|_K\in P^1(K) \mbox{ for all } K\in \cT \}\;,
\end{align*}
where $P^r(K)$ denotes the space of polynomials of degree less than or
equal to $r$ on the element $K$.
We say that the set
$\{\psi_m^n\}_{n\ge 0,m\in J_n}$ has the {\em $k$-orthogonality property},
for $k\in\{1,2\}$, with respect to the triangulations $\{\cT_\ell:\ell\ge
0\}$ if for all $\ell\ge 0$ we have
\begin{align}\label{eq:orthprop}
  & \int_D\psi_m^n(\bsx)z_\ell(\bsx) \,\rd\bsx \,=\, 0
  \quad \mbox{for all}\; n\ge\ell+k\;,\;
  m\in J_n\;,
  \;\mbox{and}\;
  z_\ell\in S^0 (D,\cT_\ell) \;,
\end{align}
and $\psi_m^n \in S^{k-1}(D,\cT_{\ell+k-1})$ for all $n\le \ell+k-1$,
$m\in J_n$, and ${\rm diam}(\supp(\psi^n_m)) \asymp h_n$.
\end{definition}

A necessary condition for \eqref{eq:orthprop} to hold is that the
functions $\psi^n_m$ for $n\ge k$ have the {\em vanishing mean property},
that is
\[
  \int_D\psi_m^n(\bsx) \,\rd\bsx \,=\, 0
  \quad\mbox{for all}\; n\ge k
  \;\mbox{and all}\; m\in J_n\;.
\]
\begin{example}[Haar Wavelets] \label{ex:Haar}
We describe here the simplest case, of Haar wavelets for a one-dimensional
domain $D=[0,a]$, with $a$ some positive integer greater than or equal to
$2$. In the Haar wavelet case we may take, for $m=0,\ldots,a-1$,
\[
  \psi_m^0(x) \,:=\,
  \begin{cases}
  1 & \mbox{for } x\in [m,m+1)\;,
  \\
  0 & \mbox{otherwise}\;,
  \end{cases}
\]
and for $n\ge 1$,
\[
  \psi_m^n(x):= d^n_m\, \psi(2^n x-2m),\quad m=0,\ldots,2^{n-1}a-1,
\]
where $d^n_m$ is a sequence of nonnegative scaling parameters, $\psi(x)$
is $1$ for $x\in[0,1)$, $-1$ for $x\in[1,2)$, and $0$ otherwise.
The family $\{\psi^n_m\}$ forms an orthogonal basis of $L^2([0,a])$ if
$d^n_m > 0$. We remark that the choice $d^n_m = 2^{(n-1)/2}$
which is
well-known to imply orthonormality of the $\psi^n_m$ in $L^2([0,a])$ is
inconsistent with (\textbf{A1}), and is therefore excluded.

For the finite element space $V_0$ we take the piecewise-linear functions
vanishing at $0$ and $a$. This space is spanned by the hat functions
centered at $1,2,\ldots,a-1$. The spaces $V_\ell$ are then the
piecewise-linear functions on $[0,a]$ vanishing at $0$ and~$a$, spanned by
the hat functions centered at multiples of $2^{-\ell}$. Correspondingly,
$\cT_\ell$ is the mesh consisting of the multiples of $2^{-\ell}$, and the
elements $K_\ell$ are the intervals of length $2^{-\ell}$ between the mesh
points.

With this definition of $\cT_\ell$, the multiresolution sequence
$\{\psi_m^n\}$ has the $k$-orthogonality property with respect to
$\cT_\ell$ with $k=1$, for all $\ell\ge 0$. For example, for $\ell=0$ and
$n=1, m=0$ we have, with $z_0\in S^0([0,a],\cT_0)$ and $c:=z_0|_{[0,1]}$,
\[
  \int_0^a\psi_0^1(x)z_0(x)\,\rd x
  \,=\,c\int_0^1\psi_0^1(x)\,\rd x
  \,=\,c\, d^1_0 \int_0^1\psi(2x)\,\rd x
  \,=\,0 \;.
\]
\end{example}
Haar wavelets do not satisfy Assumption~(\textbf{A4}), since for
(\textbf{A4}) to hold the basis functions $\psi_m^n$ need to be
Lipschitz continuous. A piecewise-linear $k$-orthogonal basis set with
$k=2$ in dimension $d=1$ is constructed, for example, in \cite{DKU}.
For detailed constructions of $k$-orthogonality basis sets with $k=2$
and $d>1$, see \cite{DKU,NguyenDiss}; for the case $k=1$ and $d>1$ see
\cite[Section 5]{BSZ}.

In the following theorem, we show that there is \emph{no truncation error
at any level} for our multi-level algorithm under $k$-orthogonality if the
dimension for truncation $s_\ell$ is chosen appropriately at each level.
This result is intrinsically linked to the linear structure in
\eqref{eq:defaxy}. To achieve this, we employ a one-to-one mapping of the
indices between the functions $\psi_j$ and $\psi^n_m$ in
\eqref{eq:relabel}: instead of ordering the functions as in
Assumption~(\textbf{A5}), we index $j$ according to a level-wise grouping
so that the functions $\{\psi_m^0\}_{m\in J_0}$ come before the functions
$\{\psi_m^1\}_{m\in J_1}$, followed by the functions $\{\psi_m^2\}_{m\in
J_2}$, and so on. Correspondingly, we employ the same index mapping
between $y_j$ and $y^n_m$ for the components of $\bsy$.

\begin{theorem}\label{thm:orthprop}
Let $\{\psi^n_m : n\ge 0, m\in J_n\}$ be a multiresolution basis set for
the domain~$D$, with $|J_n|\asymp 2^{dn}$, which has the $k$-orthogonality
property with $k\in\{1,2\}$ with respect to the triangulations
$\{\cT_\ell:\ell\ge 0\}$.
Let $\{y_j:j\ge 1\} = \{y^n_m : n\ge 0, m\in J_n\}$
denote the corresponding parameters under the level-wise
relabelling \eqref{eq:relabel} so that the parametric coefficient in
\eqref{eq:defaxy} can be represented in the form
\[ 
  a(\bsx,\bsy)
  \,=\, \bar{a}(\bsx)
  +\sum_{n=0}^\infty \sum_{m\in J_n} y^n_m \psi^n_m(\bsx)\;.
\] 
Let
\begin{equation} \label{eq:def_s_k}
  s_\ell \,:=\, \sum_{n=0}^{\ell+k-1}|J_n|\;.
\end{equation}
Then $s_\ell \asymp 2^{d\ell} \asymp M_{h_\ell}$, and for all $\ell \ge 0$
we have
\begin{equation}\label{eq:dimorth}
  u_{h_\ell} \,=\, u_{h_\ell}^{s_\ell}\;.
\end{equation}
As $\ell\to\infty$, the number of nonzero entries in the Finite Element
stiffness matrix for the parametric coefficient $a(\bsx,\bsy)$ at
meshlevel $\ell\ge 0$ for any given $\bsy\in U$ is $\calO(M_{h_\ell})$. We
assume that each of the nonzero entries can be computed in
$\calO(\log(M_{h_\ell}))$ operations, leading to a total cost of
$\calO(M_{h_\ell}\,\log(M_{h_\ell}))$ operations.
\end{theorem}
\begin{proof}
There holds $\nabla V_\ell \subseteq S^0(D,\cT_\ell)^d$ for all $\ell\geq
0$. Thus, for all $\ell\ge 0$ and for every $v_\ell,w_\ell \in V_\ell$, we
have $\nabla w_\ell \cdot \nabla v_\ell \in S^0(D,\cT_\ell)$. The
$k$-orthogonality property \eqref{eq:orthprop} therefore implies for all
$\ell \geq 0$ and for all $v_{\ell}, w_{\ell} \in V_{\ell}$
\begin{align} \label{eq:bMulti}
 b(\bsy;w_{\ell} , v_{\ell})
&\,=\,
 \int_D
 \left( \bar{a}(\bsx) + \sum_{n=0}^\infty\sum_{m\in\cI_n} y^n_m \psi^n_m(\bsx) \right)
 \nabla w_{\ell} \cdot \nabla v_{\ell} \,\rd\bsx
\nonumber
\\
 &\,=\,
 \int_D
 \left( \bar{a}(\bsx) + \sum_{n=0}^{\ell+k-1}\sum_{m\in\cI_n} y^n_m \psi^n_m(\bsx) \right)
 \nabla w_{\ell} \cdot \nabla v_{\ell} \,\rd\bsx \\
 &\,=\, b(\bsy_{\{1:{s_{\ell}}\}};w_{\ell} , v_{\ell})
 \;. \nonumber
\end{align}
The assertion \eqref{eq:dimorth} then follows from the uniqueness of the
FE solutions.

To show the assertion on the cost, for given $\bsy$ we denote by
$\matrB^{\ell}(\bsy)$ the $M_{\ell}\times M_{\ell}$ stiffness matrix of
the parametric bilinear form $b(\bsy;\cdot,\cdot)$, restricted to
$V_{\ell}\times V_{\ell}$, where $V_\ell = {\rm span}\{ \phi^\ell_i: 1\leq
i \leq M_\ell \}$, with $\phi^{\ell}_i$ denoting the nodal hat basis
functions of $S^1(D,\cT_L)$. By $k$-orthogonality of the $\psi^{n}_m$, we
have \eqref{eq:bMulti}, and for each $1\leq i,i' \leq M_{\ell} = {\rm
dim}(V_{\ell}) = \calO(2^{d\ell})$ there holds
\begin{equation}\label{eq:matrixentry}
  \matrB^{\ell}(\bsy)_{ii'}
  \,=\,
  b(\bsy_{\{1:{s_{\ell}}\}} ; \phi^\ell_i, \phi^\ell_{i'})
  \,=\,
  \int_D
  (P_{\ell+k-1} a(\bsx,\bsy))
  \nabla \phi^\ell_i \cdot \nabla \phi^\ell_{i'} \,\rd\bsx
  \;,
\end{equation}
where $P_{\ell+k-1} a(\bsx,\bsy)$ denotes the truncated expression for
$a(\bsx,\bsy)$ appearing in \eqref{eq:bMulti}.
The matrix $\matrB^{\ell}(\bsy)$ is sparse: it has, due to the local
support of the hat functions $\phi^\ell_i$ and due to the construction of
the sequence $\{ \cT_\ell \}_{\ell \geq 0}$ of meshes, at most
$\calO(M_{\ell})$ nonvanishing entries \eqref{eq:matrixentry}.

Now consider the cost for the {\em exact evaluation} of any matrix entry
$(\matrB^{\ell}(\bsy))_{ii'}\ne 0$. Given $\ell$, $i$, $i'$, and for a
given $n\le \ell+k-1$, it follows from the assumption on the support of
$\psi^n_m$ that there are only $\calO(1)$ many functions $\psi^n_m$ such
that $\int_D \psi^n_m(\bsx)\, \nabla \phi^\ell_i \cdot \nabla
\phi^\ell_{i'} \,\rd\bsx \ne 0$. Thus the cost for evaluating
$(\matrB^{\ell}(\bsy))_{ii'}\ne 0$ is $\calO(\ell+k-1)$, which yields that
the total cost for evaluating the sparse matrix is $\calO(M_\ell\,\ell) =
\calO(M_\ell\,\log(M_\ell))$ operations. \quad\qed
\end{proof}
\subsection{Key Results} \label{ssec:key}
In the error analysis of the (single level) QMC FE method, we established
in \cite{KSS1} regularity results for the parametric solutions. In the
present multi-level QMC FE error analysis, we first establish stronger
regularity of the PDE solution simultaneously with respect to both $\bsx$
and $\bsy$. The result shown is actually more general than required in
this paper: our result covers partial derivatives of arbitrary order. To
state the result, we introduce further notation: for $\bsnu =
(\nu_j)_{j\ge 1} \in \bbN_0^\bbN$, where $\bbN_0 = \bbN\cup\{0\}$, we
define $|\bsnu| := \nu_1 + \nu_2 + \cdots$, and we refer to $\bsnu$ as a
``multi-index'' and $|\bsnu|$ as the ``length'' of~$\bsnu$. By
\[
  \indx \,:=\, \{ \bsnu \in \bbN_0^\bbN \; : \; |\bsnu| < \infty \}
\]
we denote the (countable) set of all ``finitely supported'' multi-indices
(i.e., sequences of nonnegative integers for which only finitely many
entries are nonzero). For $\bsnu\in \indx$ we denote the partial
derivative of order $\bsnu\in \indx$ of $u$ with respect to $\bsy$
by
\[
  \partial^\bsnu_\bsy u \,:=\,
  \frac{\partial^{|\bsnu|}}{\partial^{\nu_1}_{y_1} \partial^{\nu_2}_{y_2} \cdots} u\;.
\]
\begin{theorem} \label{thm:reg}
Under Assumptions \textnormal{(\textbf{A1})} and
\textnormal{(\textbf{A2})}, for every $f\in V^*$, every $\bsy\in U$ and
every $\bsnu\in\indx$, the solution $u(\cdot,\bsy)$ of the parametric weak
problem \eqref{eq:paramweakprob} satisfies
\begin{equation} \label{eq:regV}
  \left\|\partial^{\bsnu}_\bsy u(\cdot,\bsy) \right\|_V
  \,\le\,
   |\bsnu|!\, \bigg(\prod_{j\ge 1} b_j^{\nu_j}\bigg) \; \frac{\|f\|_{V^*}}{a_{\min}}
   \;,
\end{equation}
where $b_j$ is as defined in \eqref{eq:defbj}.
If, in addition, $f\in H^{-1+t}(D)$ for some $0\le t\le 1$, and if
Assumption \textnormal{(\textbf{A4})} holds, then for every $\kappa
\in(0,1]$ there holds
\begin{equation} \label{eq:regZ}
  \left\|\partial^{\bsnu}_\bsy u(\cdot,\bsy)\right\|_{Z^t}
  \, \le \, C\,
  |\bsnu|!\,
  \bigg(\prod_{j\ge 1} \bar{b}_j^{\nu_j}\bigg) \|f\|_{H^{-1+t}(D)}\, \;,
\end{equation}
where
\begin{equation} \label{eq:defbarbj}
 \bar{b}_j
 \,:=\,
 b_j
 + \kappa\, C_t \left( \|\nabla \psi_j \|_{L^\infty(D)}
 + B \,\| \psi_j \|_{L^\infty(D)} \right)\;,\qquad
 j\ge 1\;,
\end{equation}
and the constants $B$ and $C_t$ are, for $0\leq t \leq 1$, defined by
\begin{equation}\label{eq:defB}
  B
  \,:=\, \frac{1}{a_{\min}}\,
  \sup_{\bsz \in U} \| \nabla a(\cdot,\bsz) \|_{L^\infty(D)}
  <\infty\;,
  \qquad
  C_t\,:=\,
  \sup_{w\in L^2(D)} \frac{\|w\|_{H^{-1+t}(D)}}{\|w\|_{L^2(D)}}
  <\infty\;.
\end{equation}
In \eqref{eq:regZ} we have $C \leq \bar{C} \kappa^{-1} $
with $\bar{C}>0$ independent of $\kappa$.
\end{theorem}
\begin{proof}
Assertion \eqref{eq:regV} was proved in \cite[Theorem 4.3]{CDS}. The proof
there was based on the observation that, for every $v\in V$, $\bsy\in U$
and $\bsnu\in\indx$ with $|\bsnu|\ne 0$, \eqref{eq:paramweakprob} implies
the recurrence
\begin{align} \label{eq:recur}
  &
  \left(a(\cdot,\bsy)\, \nabla (\partial^\bsnu_{\bsy} u(\cdot,\bsy)) \, , \, \nabla v\, \right)
  + \sum_{j\in \supp(\bsnu)} \nu_j
  \left( \psi_j\, \nabla (\partial^{\bsnu-\bse_j}_{\bsy}u(\cdot,\bsy)) \, , \, \nabla v\, \right)
  \,=\,0\;,
\end{align}
where $\bse_j\in \indx$ denotes the multiindex with entry $1$ in
position $j$ and zeros elsewhere, and where $\supp(\bsnu) := \{ j\in
\bbN: \nu_j \ne 0 \}$ denotes the ``support'' of $\bsnu$. Taking
$v(\bsx) = \partial^\bsnu_\bsy u(\bsx,\bsy)\in V$ in \eqref{eq:recur}
leads to
\begin{equation} \label{eq:recurbound}
  \| \partial^\bsnu_\bsy u(\cdot,\bsy) \|_V \,
  \le\,
  \sum_{ j \in \supp(\bsnu)} \nu_j\, b_j\,
  \| \partial^{\bsnu -\bse_j}_{\bsy}u(\cdot,\bsy) \|_V\;,
\end{equation}
from which \eqref{eq:regV} follows by induction.

Assertion \eqref{eq:regZ} was proved in \cite[Theorem 8.2]{CDS}
for the case $t=1$. For completeness we provide a proof for general $t$
here. We proceed once more by induction. The case $|\bsnu|=0$ is
precisely \eqref{eq:apriori_Z} and is already proved in
\cite[Theorem~4.1]{KSS1}. To obtain the bounds for $|\bsnu| \ne 0$, we
observe that, trivially, for every $\bsnu \in \indx$ and for every
$\bsy\in U$, the function $\partial^\bsnu_\bsy u(\cdot,\bsy)$ is the
solution of the Dirichlet problem
\begin{equation}\label{eq:Dnuybvp}
  -\nabla\cdot\left( a(\cdot,\bsy)\nabla (\partial^\bsnu_\bsy u(\cdot,\bsy))\right)
  \,=\,
  -g_\bsnu(\cdot,\bsy)
  \quad\mbox{in}\quad D\;,
  \qquad
  \partial^\bsnu_\bsy u(\cdot,\bsy)|_{\partial D} = 0\;,
\end{equation}
with
\[
  g_\bsnu(\cdot,\bsy)
  \,:=\,
  \nabla\cdot\left( a(\cdot,\bsy)\nabla (\partial^\bsnu_\bsy u(\cdot,\bsy))\right)
  \,=\,
  \nabla a(\cdot,\bsy) \cdot \nabla (\partial^\bsnu_\bsy u(\cdot,\bsy))
  +
  a(\cdot,\bsy) \Delta (\partial^\bsnu_\bsy u(\cdot,\bsy)) \;.
\]
Here, we used the identity
\begin{equation} \label{eq:nice-id}
  \nabla \cdot (\alpha(\bsx) \nabla w(\bsx))
  \,=\, \alpha(\bsx)\,\Delta w(\bsx)
  + \nabla\alpha(\bsx) \cdot \nabla w(\bsx)\;,
\end{equation}
which is valid for $\alpha\in W^{1,\infty}(D)$ and for any $w\in V$ such
that $\Delta w \in L^2(D)$.

The assertion \eqref{eq:regZ} will follow from \eqref{eq:apriori_Z},
which implies for the solution of problem \eqref{eq:Dnuybvp} the bound
\begin{equation}\label{eq:Wubound}
  \| \partial^\bsnu_\bsy u (\cdot,\bsy) \|_{Z^t}
  \,\le\,
  C\, \| g_\bsnu(\cdot,\bsy)\|_{H^{-1+t}(D)} \;.
\end{equation}
It remains to establish bounds for $\| g_\bsnu(\cdot,\bsy)
\|_{H^{-1+t}(D)}$. We recast \eqref{eq:recur} in strong form and obtain
from \eqref{eq:Dnuybvp}, for every $\bsy\in U$ and for every $v\in
H^{1-t}(D)$,
\begin{align*}
 &\left| \left( g_\bsnu(\cdot ,\bsy)\,,\, v \,\right) \right|
 \,=\,
  \left| \left(
 \nabla \cdot \left(a(\cdot , \bsy) \nabla (\partial^\bsnu_\bsy u(\cdot,\bsy)) \right)
 \, , \,
 v \,\right)\right|
 \\
 &\,=\,
  \left|
  \sum_{j\in \supp(\bsnu)} \nu_j
 \left(
  \nabla \psi_j \cdot \nabla (\partial^{\bsnu-\bse_j}_{\bsy}u(\cdot,\bsy))
  +
  \psi_j \Delta (\partial^{\bsnu-\bse_j}_{\bsy}u(\cdot,\bsy))
  \, , \,
  v \,\right)
  \right|
  \\
 &
 \,\le\, \sum_{j\in \supp(\bsnu)} \nu_j
 \left\|
  \nabla \psi_j(\cdot) \cdot \nabla (\partial^{\bsnu-\bse_j}_{\bsy}u(\cdot,\bsy))
  +
  \psi_j(\cdot) \Delta (\partial^{\bsnu-\bse_j}_{\bsy}u(\cdot,\bsy))
  \right\|_{H^{-1+t}(D)}
  \|v\|_{H^{1-t}(D)} \;.
\end{align*}
Dividing by $ \|v\|_{H^{1-t}(D)}$ and taking the supremum over all $v\in
H^{1-t}(D)$ yields
\begin{align} \label{eq:L2recur}
 \| g_\bsnu(\cdot,\bsy) \|_{H^{-1+t}(D)}
 &\,\le\, \sum_{j\in \supp(\bsnu)} \nu_j
 \left( \| \nabla \psi_j \|_{L^\infty(D)}
  \left\|\nabla (\partial^{\bsnu-\bse_j}_{\bsy}u(\cdot,\bsy))\right\|_{H^{-1+t}(D)}
 \right. \nonumber\\
 &\qquad\qquad\qquad\quad \left. + \| \psi_j \|_{L^\infty(D)}
  \|\Delta (\partial^{\bsnu-\bse_j}_{\bsy}u(\cdot,\bsy))\|_{H^{-1+t}(D)} \right)\;.
\end{align}

To bound the second term on the right-hand side of \eqref{eq:L2recur}, we
write \eqref{eq:Dnuybvp} with $\bsnu - \bse_j$ in place of $\bsnu$, for
every $\bsy \in U$, in the form
\begin{equation}\label{eq:Deltaeq}
  - a(\cdot,\bsy)
  \Delta (\partial^{\bsnu-\bse_j}_\bsy u(\cdot,\bsy))
  \,=\,
  \nabla a(\cdot,\bsy)\cdot \nabla (\partial^{\bsnu-\bse_j}_\bsy u(\cdot,\bsy))
  -
  g_{\bsnu - \bse_j}(\cdot,\bsy)
  \;,
\end{equation}
using again \eqref{eq:nice-id}. This implies, for every $\bsy\in U$, the
estimate
\begin{align*}
  & \| \Delta (\partial^{\bsnu-\bse_j}_\bsy u(\cdot,\bsy)) \|_{H^{-1+t}(D)}
  \,\le\,
  \frac{1}{a_{\min}}
  \| \mbox{RHS of \eqref{eq:Deltaeq}}\|_{H^{-1+t}(D)}
  \\
  &\,\le\,
  \frac{1}{a_{\min}}
  \left[
  \left(\sup_{\bsz\in U}
  \|\nabla a(\cdot,\bsz)\|_{L^\infty(D)}\right)
  \| \nabla (\partial^{\bsnu-\bse_j}_\bsy u(\cdot,\bsy)) \|_{H^{-1+t}(D)}
  +\| g_{\bsnu - \bse_j}(\cdot,\bsy) \|_{H^{-1+t}(D)}
  \right]
  \\
  &\,\le\,
  B\,C_t\, \| \partial^{\bsnu-\bse_j}_\bsy u(\cdot,\bsy) \|_V +\frac{1}{a_{\min}}\,
  \| g_{\bsnu - \bse_j}(\cdot,\bsy) \|_{H^{-1+t}(D)}\;,
\end{align*}
where $B$ and $C_t$ are as in \eqref{eq:defB}.
We insert this bound into \eqref{eq:L2recur} to obtain
\begin{align} \label{eq:L2recur2}
  \| g_\bsnu(\cdot,\bsy) \|_{H^{-1+t}(D)}
  &\,\le\,
  \sum_{j\in \supp(\bsnu)}
  \nu_j
  \left[ C_t
  \left( \| \nabla \psi_j \|_{L^\infty(D)} + B\, \| \psi_j \|_{L^\infty(D)}\right)
  \| \partial^{\bsnu-\bse_j}_{\bsy}u(\cdot,\bsy) \|_V
  \right.
  \nonumber\\
  &\qquad\qquad\qquad\qquad
   \left.
   +
  b_j\,
  \| g_{\bsnu - \bse_j}(\cdot,\bsy) \|_{H^{-1+t}(D)}
  \right]
  \;.
\end{align}
This recursive estimate for $\|g_{\bsnu}(\cdot,\bsy)\|_{H^{-1+t}(D)}$
has structure which is similar to the bound \eqref{eq:recurbound} for
$\| \partial^{\bsnu}_\bsy u(\cdot,\bsy) \|_V$. We therefore multiply
\eqref{eq:L2recur2} by $\kappa >0$ and add it to \eqref{eq:recurbound}
to obtain
\begin{align} \label{eq:recurboundW}
  &\| \partial^{\bsnu}_{\bsy}u(\cdot,\bsy) \|_{V}
  +\kappa
  \| g_\bsnu(\cdot,\bsy) \|_{H^{-1+t}(D)} \nonumber\\
  &\,\le\,
  \sum_{j\in \supp(\bsnu)}
  \nu_j\,
  b_j\,
  \left[
  \| \partial^{\bsnu-\bse_j}_{\bsy}u(\cdot,\bsy) \|_{V}
  +\kappa
  \| g_{\bsnu-\bse_j}(\cdot,\bsy) \|_{H^{-1+t}(D)}
  \right]
  \nonumber\\
  &\qquad +
  \sum_{j\in \supp(\bsnu)}
  \nu_j\,\kappa\,C_t\,
  \left( \| \nabla \psi_j \|_{L^\infty(D)} + B \| \psi_j \|_{L^\infty(D)}\right)
  \| \partial^{\bsnu-\bse_j}_{\bsy}u(\cdot,\bsy) \|_{V}
  \nonumber\\
  &\,\le\,
  \sum_{j\in \supp(\bsnu)}
  \nu_j\,
  \bar{b}_j \,
  \left[
  \| \partial^{\bsnu-\bse_j}_{\bsy}u(\cdot,\bsy) \|_{V}
  +\kappa
  \| g_{\bsnu-\bse_j}(\cdot,\bsy) \|_{H^{-1+t}(D)}
  \right],
\end{align}
where $\bar{b}_j$ is as in \eqref{eq:defbarbj}. By Assumption
(\textbf{A4}), we have $\sum_{j\ge 1} \bar{b}_j < \infty$ for any choice
of $\kappa>0$ and for any $B$.

To establish \eqref{eq:regZ} it remains to observe that the estimate
\eqref{eq:recurboundW} has the same structure as \eqref{eq:recurbound},
with the sequence $\{\bar{b}_j\}$ in place of $\{b_j\}$. For $|\bsnu| =
0$, we find using \eqref{eq:apriori_V} of Theorem~\ref{thm:weak} and
$g_\bszero = -f$ that
$$
\| u(\cdot,\bsy) \|_V + \kappa\, \| g_\bszero \|_{H^{-1+t}(D)}
\,\leq\,
\frac{1}{a_{\min}} \| f \|_{V^*} + \kappa\, \| f \|_{H^{-1+t}(D)}
\;.
$$
The same induction argument used to establish \eqref{eq:regV} applied
to the recursive estimate \eqref{eq:recurboundW} implies for all $\bsnu\in
\indx$, for every $\bsy \in U$ and for every $\kappa\in (0,1]$
\begin{align*}
  \kappa\, \| g_\bsnu(\cdot,\bsy) \|_{H^{-1+t}(D)}
  &\,\le\,
  \| \partial^{\bsnu}_{\bsy}u(\cdot,\bsy) \|_{V} + \kappa\,
  \| g_\bsnu(\cdot,\bsy) \|_{H^{-1+t}(D)}
\\
  &\,\le\,
  |\bsnu|!\,
  \bigg(\prod_{j\ge 1} \bar{b}_j^{\nu_j}\bigg)\,
  \left( \frac{\tilde{C}_t}{a_{\min}} +
  \kappa \right) \| f \|_{H^{-1+t}(D)}\;,
\end{align*}
where $\tilde{C}_t:=
 \sup_{w\in H^{-1+t}(D)} (\|w \|_{H^{-1}(D)}/\|w\|_{H^{-1+t}(D)})
 <\infty$.
Now \eqref{eq:regZ} follows from \eqref{eq:Wubound}. \quad\qed
\end{proof}


To bound the first term in \eqref{eq:key} we need Theorem~\ref{thm:AN}
below. We shall make use of the following lemma which can be proved by
induction. We use the convention that an empty product is $1$.
\begin{lemma} \label{lem:recur}
Given non-negative numbers $(\beta_j)_{j\in\bbN}$, let
$(\bbA_\setv)_{\setv\subset\bbN}$ and $(\bbB_\setv)_{\setv\subset\bbN}$ be
non-negative real numbers satisfying the inequality
\[
  \bbA_\setv
  \,\le\, \sum_{k\in\setv} \beta_k\, \bbA_{\setv\setminus\{k\}} + \bbB_\setv
  \quad\mbox{for any $\setv\subset\bbN$ (including $\setv=\emptyset$)}.
\]
Then we have
\[
  \bbA_\setv
  \,\le\, \sum_{\setw\subseteq\setv} |\setw|!\, \bigg(\prod_{j\in\setw} \beta_j\bigg)\, \bbB_{\setv\setminus\setw}\;.
\]
\end{lemma}
\begin{theorem} \label{thm:AN}
Under Assumptions \textnormal{(\textbf{A1})}, \textnormal{(\textbf{A2})},
\textnormal{(\textbf{A4})}, and \textnormal{(\textbf{A6})}, for every
$f\in H^{-1+t}(D)$ with $0\le t\le 1$, every $G\in H^{-1+t'}(D)$ with $0
\le t'\le 1$, every $\kappa\in (0,1]$, and every $s\in \bbN$, we have
\begin{align*}
  &\|G(u^s - u^s_h)\|_{\calW_{s,\bsgamma}} \\
  &\,\le\, C\, h^{\tau}\, a_{\max}\, \|f\|_{H^{-1+t}(D)}\, \|G\|_{H^{-1+t'}(D)}
  \left(\sum_{\setu\subseteq\{1:s\}}
  \frac{[(|\setu|+3)!]^2\prod_{j\in\setu} \bar{b}_j^2}{\gamma_\setu}
  \right)^{1/2}\;,
\end{align*}
where $0 \le \tau := t + t' \le 2$, $\bar{b}_j$ is
defined in \eqref{eq:defbarbj},
and where the constant $C>0$ is independent of $s$.
\end{theorem}

\begin{proof}
Let $g\in H^{-1+t'}(D)$ denote the representer of $G\in H^{-1+t'}(D)$.
Here, for $0<t'<1$, we have $H^{-1+t'}(D) = (H^{1-t'}_0(D))^*$ with
duality taken with respect to the ``pivot'' space $L^2(D)\simeq
(L^2(D))^*$, and with $H^{1-t'}_0(D) := (H^1_0(D),L^2(D))_{1-t'}$ defined
by interpolation. Then, with $(\cdot,\cdot)$ denoting the $H^{-1+t'}(D)
\times H^{1-t'}_0(D)$ duality pairing, we have that $G(w)=(g,w)$ for $w\in
H^{1-t^\prime}_0(D)$.

For all $\bsy\in U$, we then define $v^g(\cdot,\bsy)\in V$ and
$v_h^g(\cdot,\bsy)\in V_h$ by
\begin{align*}
  b(\bsy; w, v^g(\cdot,\bsy)) &\,=\, ( g,w )
  \;\qquad\forall w\in V\;,
  \\
  b(\bsy; w_h, v_h^g(\cdot,\bsy)) &\,=\, ( g,w_h )
  \qquad\forall w_h\in V_h\;,
\end{align*}
so that $v^g$ and $v^g_h$ are the exact and FE solutions if $f$ is
replaced by $g$.
Taking $w = u(\cdot,\bsy) - u_h(\cdot,\bsy)$, we have
\begin{align*}
  G(u(\cdot,\bsy)-u_h(\cdot,\bsy))
  &\,=\, ( g, u(\cdot,\bsy)-u_h(\cdot,\bsy) ) \\
  &\,=\, b(\bsy; u(\cdot,\bsy)-u_h(\cdot,\bsy), v^g(\cdot,\bsy))\\
  &\,=\, b(\bsy; u(\cdot,\bsy)-u_h(\cdot,\bsy), v^g(\cdot,\bsy) - v_h^g(\cdot,\bsy))\;,
\end{align*}
where we used Galerkin orthogonality
$b(\bsy;u(\cdot,\bsy)-u_h(\cdot,\bsy), v_h^g(\cdot,\bsy))= 0$.

Using the definitions of the bilinear form $b(\bsy;\cdot,\cdot)$ and the
norm $\|\cdot\|_{\calW_{s,\bsgamma}}$, we obtain
\begin{align} \label{eq:want_1}
  &\| G(u^s-u_h^s)\|_{\calW_{s,\bsgamma}} \nonumber\\
  &\,=\,
  \left(\sum_{\setu\subseteq\{1:s\}} \frac{1}{\gamma_\setu}
  \int_{[-\frac{1}{2},\frac{1}{2}]^{|\setu|}}
  \left|
  \int_{[-\frac{1}{2},\frac{1}{2}]^{s-|\setu|}}
  r_\setu(\bsy_\setu;\bsy_{-\setu};\bszero)
  \,\rd\bsy_{-\setu}
  \right|^2
  \rd\bsy_\setu\right)^{1/2}\;,
\end{align}
where we define for all $\bsy\in U$
\[
  r_\setu(\bsy)
  \,:=\, \int_D \frac{\partial^{|\setu|}}{\partial \bsy_\setu}
  \Big( a(\bsx,\bsy)\,\nabla (u-u_h)(\bsx,\bsy)\cdot\nabla (v^g - v_h^g)(\bsx,\bsy)\Big)\,\rd\bsx\;.
\]
For the remainder of this proof, we will use the short-hand notation
$\partial_\setu$ for the mixed first partial derivatives with respect to
the variables $y_j$ for $j\in\setu$. From the definition of $a(\bsx,\bsy)$
we see that
\begin{align*}
  r_\setu(\bsy)
  &\,=\, \int_D a(\bsx,\bsy)\, \partial_\setu
  \Big(\nabla (u-u_h)(\bsx,\bsy)\cdot\nabla (v^g - v_h^g)(\bsx,\bsy)\Big)\,\rd\bsx \\
  &\qquad
  + \sum_{k\in\setu} \int_D \psi_{k}(\bsx)\, \partial_{\setu\setminus\{k\}}
  \Big(\nabla (u-u_h)(\bsx,\bsy)\cdot\nabla (v^g - v_h^g)(\bsx,\bsy)\Big) \,\rd\bsx\\
 &\,=\, \int_D a(\bsx,\bsy) \sum_{\setv\subseteq\setu}
  \nabla \partial_\setv (u-u_h)(\bsx,\bsy) \cdot
  \nabla \partial_{\setu\setminus\setv} (v^g - v_h^g)(\bsx,\bsy) \,\rd\bsx\\
  &\qquad + \sum_{{k}\in\setu} \int_D \psi_{{k}}(\bsx)\,
  \sum_{\setv\subseteq\setu\setminus\{{k}\}}
  \nabla \partial_\setv (u-u_h)(\bsx,\bsy) \cdot
  \nabla \partial_{(\setu\setminus\{{k}\})\setminus\setv} (v^g - v_h^g)(\bsx,\bsy) \,\rd\bsx\;,
\end{align*}
where in both terms we used the product rule $\partial_\setu (AB) =
\sum_{\setv\subseteq\setu} (\partial_\setv
A)(\partial_{\setu\setminus\setv} B)$.
Thus
\begin{align} \label{eq:int_Ru}
  \left| r_\setu(\bsy) \right|
  &\,\le\, a_{\max}\, \sum_{\setv\subseteq\setu}
  \|\partial_\setv (u-u_h)(\cdot,\bsy)\|_V \,
  \|\partial_{\setu\setminus\setv} (v^g - v_h^g)(\cdot,\bsy)\|_V
  \\
  &\qquad + \sum_{k \in\setu} \|\psi_{k}\|_{L^\infty(D)}\,
  \sum_{\setv\subseteq\setu\setminus\{k\}}
  \|\partial_\setv (u-u_h)(\cdot,\bsy)\|_V \,
  \|\partial_{(\setu\setminus\{k\})\setminus\setv} (v^g - v_h^g)(\cdot,\bsy)\|_V\;. \nonumber
\end{align}

To continue, we need to obtain an estimate for $\|\partial_\setv
(u-u_h)(\cdot,\bsy)\|_V$. Let $\calI:V\to V$ denote the identity operator,
and for $\bsy\in U$ let $\calP_h = \calP_h(\bsy) :V\to V_h$ denote the
parametric FE projection defined by
\begin{equation} \label{eq:Ph}
b(\bsy; \calP_h w, z_h) = b(\bsy;w,z_h) \qquad\forall\, w\in V, \; z_h\in V_h\; .
\end{equation}
Then we have $u_h= \calP_h u\in V_h$ and $\partial_\setv u_h\in V_h$, and
hence $(\calI - \calP_h)\partial_\setv u_h = 0$. Thus
\begin{align} \label{eq:new_1}
  \|\partial_\setv (u-u_h)(\cdot,\bsy)\|_V
  & \,=\, \| \calP_h \partial_\setv (u-u_h)(\cdot,\bsy) + (\calI - \calP_h) \partial_\setv u(\cdot,\bsy)\|_V \nonumber\\
  & \,\le\, \| \calP_h \partial_\setv (u-u_h)(\cdot,\bsy)\|_V \,+\, \|(\calI - \calP_h) \partial_\setv u(\cdot,\bsy)\|_V\;.
\end{align}
Recall that Galerkin orthogonality gives
$b(\bsy;u(\cdot,\bsy)-u_h(\cdot,\bsy), z_h)= 0$ for all $z_h\in V_h$. Upon
differentiating with respect to $\bsy_\setv$, we obtain for all $z_h\in
V_h$ that
\begin{align} \label{eq:new_2}
 &\int_D a(\bsx,\bsy)\, \nabla \big(\partial_\setv (u-u_h)(\bsx,\bsy)\big)\cdot\nabla z_h(\bsx)\,\rd\bsx
 \nonumber \\
 &\,=\,
 - \sum_{k\in\setv} \int_D \psi_{k}(\bsx)\,
  \nabla \partial_{\setv\setminus\{k\}} (u-u_h)(\bsx,\bsy)\cdot\nabla z_h(\bsx)\,\rd \bsx\;.
\end{align}
Using again the definition \eqref{eq:Ph} of $\calP_h$, we may replace
$\partial_\setv (u-u_h)$ on the left-hand side of \eqref{eq:new_2} by
$\calP_h
\partial_\setv (u-u_h)$. Taking $z_h = \calP_h
\partial_\setv (u-u_h)(\cdot,\bsy)$, we then obtain
\begin{align*}
 & a_{\min}\, \|\calP_h \partial_\setv (u-u_h)(\cdot,\bsy)\|_V^2 \\
 &\,\le\,
 \sum_{k\in\setv} \| \psi_{k}\|_{L^\infty(D)}\, \|\partial_{\setv\setminus\{k\}} (u-u_h)(\cdot,\bsy)\|_V\,
 \|\calP_h \partial_\setv (u-u_h)(\cdot,\bsy)\|_V,
\end{align*}
which in turn yields
\begin{align} \label{eq:new_3}
 \|\calP_h \partial_\setv (u-u_h)(\cdot,\bsy)\|_V
  \,\le\, \sum_{k\in\setv} b_k\, \|\partial_{\setv\setminus\{k\}} (u-u_h)(\cdot,\bsy)\|_V\;.
\end{align}
Substituting \eqref{eq:new_3} into \eqref{eq:new_1} gives
\begin{align*} 
  \|\partial_\setv (u-u_h)(\cdot,\bsy)\|_V
  &\,\le\, \sum_{k\in\setv} b_k\, \|\partial_{\setv\setminus\{k\}} (u-u_h)(\cdot,\bsy)\|_V
  + \|(\calI - \calP_h) \partial_{\setv} u(\cdot,\bsy)\|_V\;,
\end{align*}
from which we conclude using Lemma~\ref{lem:recur} that
\begin{align*} 
  \|\partial_\setv (u-u_h)(\cdot,\bsy)\|_V
  &\,\le\, \sum_{\setw\subseteq\setv}
  |\setw|!\, \bigg(\prod_{k\in\setw} b_k\bigg)\,
  \|(\calI - \calP_h) \partial_{\setv\setminus\setw} u(\cdot,\bsy)\|_V\;.
\end{align*}
Next we use the FE estimate that for all $\bsy\in U$ and $w\in V$ we have
$\|(\calI-\calP_h(\bsy))w\|_V \le C\, h^t\, \|w\|_{Z^t}$ (in particular,
this implies \eqref{eq:FE1} in Theorem~\ref{thm:FE}). This yields
\begin{align} \label{eq:nice_1}
  \|\partial_\setv (u-u_h)(\cdot,\bsy)\|_V
  &\,\le\, C\, h^t \sum_{\setw\subseteq\setv}
  |\setw|!\, \bigg(\prod_{k\in\setw} b_k\bigg)\,
  \|\partial_{\setv\setminus\setw} u(\cdot,\bsy)\|_{Z^t} \nonumber\\
  &\,\le\, C\, h^t\,\|f\|_{H^{-1+t}(D)} \sum_{\setw\subseteq\setv}
  |\setw|!\, \bigg(\prod_{k\in\setw} b_k\bigg)\,
  |\setv\setminus\setw|!\,\bigg(\prod_{j\in\setv\setminus\setw} \bar{b}_j\bigg) \nonumber\\
  &\,\le\, C\, h^t\,\|f\|_{H^{-1+t}(D)} (|\setv|+1)!\,
  \prod_{j\in\setv}\bar{b}_j\;,
\end{align}
where the second inequality follows from \eqref{eq:regZ} in
Theorem~\ref{thm:reg}, and the final step follows from $b_k\le \bar{b}_k$
and the identity $\sum_{\setw\subseteq\setv}
|\setw|!\,|\setv\setminus\setw|! = (|\setv|+1)!$. Throughout, $C>0$
denotes a generic constant.

Similarly, 
with $f$ replaced by $g$, $u$ replaced by $v^g$, $u_h$ replaced by
$v^g_h$, $t$ replaced by~$t'$, and $\setv$ replaced by
$\setu\setminus\setv$, we obtain
\begin{align} \label{eq:nice_2}
  \|\partial_{\setu\setminus\setv} (v^g-v_h^g)(\cdot,\bsy)\|_V
  &\,\le\, C\,h^{t'}\,\|g\|_{H^{-1+t'}(D)}\, (|\setu\setminus\setv|+1)!\,
  \prod_{j\in\setu\setminus\setv} \bar{b}_j\;.
\end{align}
Using \eqref{eq:nice_1} and \eqref{eq:nice_2} and the identity
$\sum_{\setv\subseteq\setu} (|\setv|+1)!\,(|\setu\setminus\setv|+1)! =
(|\setu|+3)!/6$, we obtain from \eqref{eq:int_Ru}
\begin{align*}
  \left| r_\setu(\bsy) \right|
  &\,\le\, C\,h^{t+t'}\,a_{\max}\,\|f\|_{H^{-1+t}(D)}\,\|g\|_{H^{-1+t'}(D)}\,
  \tfrac{1}{6}(|\setu|+3)!\,
  \prod_{j\in\setu} \bar{b}_j \\
  &\qquad + C\,h^{t+t'}\,\|f\|_{H^{-1+t}(D)}\,\|g\|_{H^{-1+t'}(D)}
  \sum_{k\in\setu} \|\psi_{k}\|_{L^\infty(D)}\, \tfrac{1}{6}(|\setu|+2)!\,\prod_{j\in\setu\setminus\{k\}}\,\bar{b}_{j} \\
  &\,\le\, C\,h^{t+t'}\,a_{\max}\,\|f\|_{H^{-1+t}(D)}\,\|G\|_{H^{-1+t'}(D)}\,
  (|\setu|+3)!\, \prod_{j\in\setu} \bar{b}_j\;,
\end{align*}
where we used the estimate $\|\psi_{k}\|_{L^\infty(D)} = a_{\min}\,
b_{k}\le a_{\max}\, \bar{b}_{k}$. Substituting this estimate into
\eqref{eq:want_1} completes the proof. \quad\qed
\end{proof}

As we remarked earlier, if $k$-orthogonality \eqref{eq:orthprop} does not
hold and if $s_\ell > s_{\ell-1}$, the second term in \eqref{eq:key} is
generally nonzero. We estimate it in the following result.

\begin{theorem} \label{thm:trunc-ML-new}
Under Assumptions \textnormal{(\textbf{A1})} and
\textnormal{(\textbf{A2})}, for every $f\in V^*$, every $G\in V^*$, every
$h>0$, and every $\ell\ge 1$,
\begin{align} \label{eq:Thm8a-1}
  &\|G(u_h^{s_\ell}-u_h^{s_{\ell-1}})\|_{\calW_{s_\ell,\bsgamma}} \nonumber\\
  &\,\le\,
  \frac{\|f\|_{V^*}\,\|G\|_{V^*}}{a_{\min}}
  \Bigg[
  \bigg(\frac{1}{2} \sum_{j=s_{\ell-1}+1}^{s_\ell} b_j\bigg)^2
  \sum_{\setu\subseteq\{1:s_{\ell-1}\}} \!\!
  \frac{[(|\setu|+1)!]^2\,\prod_{j\in\setu} b_j^2}{\gamma_\setu} \nonumber
  \\
  &\qquad\qquad\qquad\qquad\qquad\qquad\qquad
  + \sum_{\satop{\setu\subseteq\{1:s_{\ell}\}}{\setu\cap\{s_{\ell-1}+1:s_\ell\}\ne\emptyset}}\!\!
  \frac{[(|\setu|)!]^2\,\prod_{j\in\setu} b_j^2}{\gamma_\setu}
  \Bigg]^{1/2},
\end{align}
where $b_j$ is defined in \eqref{eq:defbj}. In addition, if $s_{\ell-1}\ne
s_\ell$, and Assumptions \textnormal{(\textbf{A3})} and
\textnormal{(\textbf{A5})} hold, and the weights $\gamma_\setu$ are such
that
\begin{equation} \label{eq:nasty}
  \sum_{\satop{\setu\subseteq\{1:s_{\ell}\}}{\setu\cap\{s_{\ell-1}+1:s_\ell\}\ne\emptyset}}\!\!
  \frac{[(|\setu|)!]^2\,\prod_{j\in\setu} b_j^2}{\gamma_\setu}
  \,\le\, C\,
  s_{\ell-1}^{-2\alpha}
  \sum_{\setu\subseteq\{1:s_{\ell}\}} \!\!
  \frac{[(|\setu|+n)!]^2\,\prod_{j\in\setu} b_j^2}{\gamma_\setu} 
\end{equation}
for some $\alpha>0$ and integer $n\ge 1$, then
\begin{align} \label{eq:Thm8a-2}
  &\|G(u_h^{s_\ell}-u_h^{s_{\ell-1}})\|_{\calW_{s_\ell,\bsgamma}} \nonumber \\
  &\,\le\, \tilde{C}\, \|f\|_{V^*}\,\|G\|_{V^*}\,
  s_{\ell-1}^{-\min(1/p-1,\alpha)}
  \Bigg(\sum_{\setu\subseteq\{1:s_{\ell}\}} \!\!
  \frac{[(|\setu|+n)!]^2\,\prod_{j\in\setu} b_j^2}{\gamma_\setu}
  \Bigg)^{1/2}.
\end{align}
Both $C,\tilde{C}>0$ are generic constants which are independent of
$s_\ell$ and $s_{\ell-1}$.
\end{theorem}

\begin{proof}
As in the proof of Theorem~\ref{thm:AN}, we will use the short-hand
notation $\partial_\setu$ for the mixed first partial derivatives with
respect to the variables $y_j$ for $j\in\setu$. For any $\bsy\in U$,
$u_h^{s_\ell}(\cdot,\bsy)$ and $u_h^{s_{\ell-1}}(\cdot,\bsy)$ are the
solutions of the variational problems:
\begin{align}
  ( a^{s_\ell}(\cdot,\bsy) \nabla u_h^{s_\ell}(\cdot,\bsy) , \nabla z_h )
  &\,=\, (f,z_h) \qquad \forall\, z_h \in V_h\;, \label{eq:one} \\
  ( a^{s_{\ell-1}}(\cdot,\bsy) \nabla u_h^{s_{\ell-1}}(\cdot,\bsy),\nabla z_h )
  &\,=\, (f,z_h) \qquad\forall\, z_h \in V_h\;. \label{eq:two}
\end{align}
To estimate
$\|G(u_h^{s_\ell}-u_h^{s_{\ell-1}})\|_{\calW_{s_\ell,\bsgamma}}$, we make
use of the inequality
\[
  |\partial_\setu (G(u_h^{s_\ell}-u_h^{s_{\ell-1}}))(\bsy)|
  \,\le\, \|G\|_{V^*}\, \|\partial_\setu (u_h^{s_\ell}-u_h^{s_{\ell-1}})(\cdot,\bsy)\|_V.
\]

If $\setu\cap\{s_{\ell-1}+1:s_\ell\}\ne\emptyset$, then it follows from
\eqref{eq:regV} of Theorem~\ref{thm:reg} that
\begin{equation} \label{eq:piece1a}
  \|\partial_\setu (u_h^{s_\ell}-u_h^{s_{\ell-1}})(\cdot,\bsy)\|_V
  \,=\, \|\partial_\setu u_h^{s_\ell}(\cdot,\bsy)\|_V
  \,\le\, |\setu|!\bigg(\prod_{j\in\setu} b_j\bigg) \frac{\|f\|_{V^*}}{a_{\min}}\;.
\end{equation}
On the other hand, if $\setu\subseteq\{1:s_{\ell-1}\}$ then we subtract
\eqref{eq:two} from \eqref{eq:one} to obtain the equation $(
a^{s_\ell}(\cdot,\bsy)\nabla u_h^{s_\ell}(\cdot,\bsy)-
a^{s_{\ell-1}}(\cdot,\bsy) \nabla u_h^{s_{\ell-1}}(\cdot,\bsy), \nabla
z_h) \,=\, 0$ for all $z_h \in V_h$, or equivalently,
\begin{align*}
  &( a^{s_\ell}(\cdot,\bsy) \nabla(u_h^{s_\ell}(\cdot,\bsy)-u_h^{s_{\ell-1}}(\cdot,\bsy)),\nabla z_h )
  \\
  &\,=\, - ( (a^{s_\ell}(\cdot,\bsy)-a^{s_{\ell-1}}(\cdot,\bsy))\nabla u_h^{s_{\ell-1}}(\cdot,\bsy),\nabla z_h )
  \qquad\forall\, z_h\in V_h\;.
\end{align*}
Upon differentiating with respect to $\bsy_\setu$ for
$\setu\subseteq\{1:s_{\ell-1}\}$, we obtain
\begin{align*}
  &\int_D a^{s_\ell}(\bsx,\bsy)\,\nabla \partial_\setu (u_h^{s_\ell}-u_h^{s_{\ell-1}})(\bsx,\bsy)
  \cdot\nabla z_h(\bsx)\,\rd\bsx \\
  &\,=\, - \sum_{k\in\setu} \int_D \psi_k(\bsx)\,\nabla \partial_{\setu\setminus\{k\}} (u_h^{s_\ell}-u_h^{s_{\ell-1}})(\bsx,\bsy)
  \cdot\nabla z_h(\bsx)\,\rd\bsx \\
  &\qquad
  - \int_D \bigg(\sum_{j=s_{\ell-1}+1}^{s_\ell} \psi_j(\bsx)\,y_j\bigg)
  \nabla \partial_{\setu} u_h^{s_{\ell-1}}(\bsx,\bsy) \cdot\nabla z_h(\bsx)\,\rd\bsx\;.
\end{align*}
Taking $z_h = \partial_\setu (u_h^{s_\ell}-u_h^{s_{\ell-1}})(\cdot,\bsy)$,
we get using similar steps to those for obtaining~\eqref{eq:new_3},
\begin{align*}
  &\|\partial_\setu (u_h^{s_\ell}-u_h^{s_{\ell-1}})(\cdot,\bsy)\|_V \\
  &\,\le\, \sum_{k\in\setu} b_k\, \|\partial_{\setu\setminus\{k\}} (u_h^{s_\ell}-u_h^{s_{\ell-1}})(\cdot,\bsy)\|_V
  + \bigg(\frac{1}{2} \sum_{j=s_{\ell-1}+1}^{s_\ell} b_j\bigg)
  \|\partial_\setu u^{s_{\ell-1}}_h(\cdot,\bsy)\|_{V}\;.
\end{align*}
It then follows from Lemma~\ref{lem:recur} that
\begin{align} \label{eq:piece1b}
  &\|\partial_\setu (u_h^{s_\ell}-u_h^{s_{\ell-1}})(\cdot,\bsy)\|_V
  \,\le\, \bigg(\frac{1}{2} \sum_{j=s_{\ell-1}+1}^{s_\ell} b_j\bigg)
  \sum_{\setv\subseteq\setu} |\setv|! \bigg(\prod_{j\in\setv} b_j\bigg)\,
  \|\partial_{\setu\setminus\setv} u^{s_{\ell-1}}_h(\cdot,\bsy)\|_{V} \nonumber\\
  &\,\le\, \bigg(\frac{1}{2} \sum_{j=s_{\ell-1}+1}^{s_\ell} b_j\bigg)
  \sum_{\setv\subseteq\setu} |\setv|! \bigg(\prod_{j\in\setv} b_j\bigg)\,
  |\setu\setminus\setv|! \bigg(\prod_{j\in\setu\setminus\setv} b_j\bigg)\,
  \frac{\|f\|_{V^*}}{a_{\min}} \nonumber\\
  &\,\le\, \bigg(\frac{1}{2} \sum_{j=s_{\ell-1}+1}^{s_\ell} b_j\bigg)
  (|\setu|+1)! \bigg(\prod_{j\in\setu} b_j\bigg)\,
  \frac{\|f\|_{V^*}}{a_{\min}}\;,
\end{align}
where we used again \eqref{eq:regV} of Theorem~\ref{thm:reg} and the
identity $\sum_{\setv\subseteq\setu} |\setv|!\,|\setu\setminus\setv|! =
(|\setu|+1)!$.

Combining \eqref{eq:piece1a} and \eqref{eq:piece1b}, we conclude that
\begin{align*} 
  &\|G(u_h^{s_\ell}-u_h^{s_{\ell-1}})\|_{\calW_{s_\ell,\bsgamma}}^2 \nonumber\\
  &\,\le\, \sum_{\setu\subseteq\{1:s_{\ell-1}\}} \frac{1}{\gamma_\setu}
  \bigg[ \|G\|_{V^*}\,
  \bigg(\frac{1}{2} \sum_{j=s_{\ell-1}+1}^{s_\ell} b_j\bigg)
  (|\setu|+1)!\, \bigg(\prod_{j\in\setu} b_j\bigg)\,
  \frac{\|f\|_{V^*}}{a_{\min}}
  \bigg]^2 \nonumber\\
  &\qquad
  + \sum_{\satop{\setu\subseteq\{1:s_{\ell}\}}{\setu\cap\{s_{\ell-1}+1:s_\ell\}\ne\emptyset}}
  \frac{1}{\gamma_\setu}
  \bigg[ \|G\|_{V^*}\,
  |\setu|!\, \bigg(\prod_{j\in\setu} b_j\bigg)\,
  \frac{\|f\|_{V^*}}{a_{\min}}
  \bigg]^2,
\end{align*}
which yields the estimate \eqref{eq:Thm8a-1}. The estimate
\eqref{eq:Thm8a-2} then follows directly from \eqref{eq:stechkin} and the
condition \eqref{eq:nasty}. \qed
\end{proof}

\subsection{Error Analysis of the Multi-level QMC FE Algorithm (Continued)}
\label{ssec:err2}

We are now ready to estimate the two terms in \eqref{eq:key} for $\ell\ne
0$. To bound the first term, we use the triangle inequality
\begin{align*}
  \|G(u_{h_\ell}^{s_\ell}-u_{h_{\ell-1}}^{s_\ell})\|_{\calW_{s_\ell,\bsgamma}}
  \,\le\, \| G(u^{s_\ell} - u_{h_\ell}^{s_\ell})
    \|_{\calW_{s_\ell,\bsgamma}} + \| G(u^{s_\ell} -
    u_{h_\ell-1}^{s_\ell}) \|_{\calW_{s_\ell,\bsgamma}}\;,
\end{align*}
and then apply Theorem~\ref{thm:AN} to both terms on the right-hand side.
If $k$-orthogonality \eqref{eq:orthprop} does not hold and if $s_\ell\ne
s_{\ell-1}$, we assume \eqref{eq:nasty} holds and bound the second term in
\eqref{eq:key} using \eqref{eq:Thm8a-2} of Theorem~\ref{thm:trunc-ML-new}.
For the $\ell = 0$ term in \eqref{eq:T2-bound}, we use the estimate
\begin{align*}
  &\|G(u_{h_0}^{s_0})\|_{\calW_{s_0,\bsgamma}} \\
  &\,\le\, \left(\sum_{\setu\subseteq\{1:s_0\}} \frac{1}{\gamma_\setu}
  \int_{[-\frac{1}{2},\frac{1}{2}]^{|\setu|}}
  \bigg|
  \int_{[-\frac{1}{2},\frac{1}{2}]^{s_0-|\setu|}}
  \|G\|_{V^*}
  \bigg\|\frac{\partial^{|\setu|}u_{h_0}^{s_0}}{\partial \bsy_\setu}
  (\cdot,(\bsy_\setu;\bsy_{-\setu};\bszero))
  \bigg\|_V
  \rd\bsy_{-\setu}
  \bigg|^2
  \rd\bsy_\setu \right)^{1/2}
  \\
  &\,\le\,
  \frac{\|f\|_{V^*}\,\|G\|_{V^*}}{a_{\min}}
  \Bigg(\sum_{\setu\subseteq\{1:s_0\}}
  \frac{(|\setu|!)^2\,\prod_{j\in\setu} b_j^2}{\gamma_\setu}
  \Bigg)^{1/2}\;,
\end{align*}
which follows from an adaptation of \eqref{eq:regV} from
Theorem~\ref{thm:reg}.
Combining these estimates with \eqref{eq:error-bound},
\eqref{eq:T1-bound}, \eqref{eq:T2-bound}, \eqref{eq:key}, and
\eqref{eq:stechkin}, we obtain
\begin{align*} 
  &\bbE[|I(G(u)) - Q_*^L(\cdot;G(u))|^2] \nonumber\\
  &\le C \Biggr(
  \left[
  h_L^{\tau}\, \|f\|_{H^{-1+t}(D)}\,\|G\|_{H^{-1+t'}(D)}
  \,+\, \theta_L\, s_L^{-2(1/p-1)}\, \|f\|_{V^*}\,\|G\|_{V^*}
  \right]^2
  \nonumber\\
  &\quad + \Bigg(\sum_{\emptyset\ne\setu\subseteq\{1:s_0\}}
  \!\!\!\!\gamma_\setu^\lambda\, [\rho(\lambda)]^{|\setu|}\Bigg)^{1/\lambda}\,
  [\varphi(N_0)]^{-1/\lambda}\,
  \|f\|_{V^*}^2\, \|G\|_{V^*}^2
  \sum_{\setu\subseteq\{1:s_0\}}\!\!\!\!\!\!
  \frac{(|\setu|!)^2\prod_{j\in\setu} b_j^2}{\gamma_\setu} 
\end{align*} 
\begin{align*} 
  &\quad + \sum_{\ell=1}^L
  \Bigg(\sum_{\emptyset\ne\setu\subseteq\{1:s_\ell\}}
  \!\!\!\!\gamma_\setu^\lambda\, [\rho(\lambda)]^{|\setu|}\Bigg)^{1/\lambda}\,
  [\varphi(N_\ell)]^{-1/\lambda}\, \nonumber\\
  &\qquad \cdot \Biggr[
  h_{\ell-1}^{\tau}\,
  \|f\|_{H^{-1+t}(D)}\, \|G\|_{H^{-1+t'}(D)}
  \left(
  \sum_{\setu\subseteq\{1:s_\ell\}}\!\!\!\!\!\!
  \frac{[(|\setu|+3)!]^2\prod_{j\in\setu} \bar{b}_j^2}{\gamma_\setu}
  \right)^{1/2}
  \nonumber\\
  &\qquad
  + \theta_{\ell-1}\,s_{\ell-1}^{-\min(1/p-1,\alpha)}
  \|f\|_{V^*}\,\|G\|_{V^*}
  \left(
  \sum_{\setu\subseteq\{1:s_\ell\}}\!\!\!\!\!\!
  \frac{[(|\setu|+n)!]^2\prod_{j\in\setu} b_j^2}{\gamma_\setu}
  \right)^{1/2}
  \Biggr]^2 \Biggr)\,,
\end{align*}
where we introduced the parameters $\theta_{\ell-1}\in \{0,1\}$ for
each level, analogously to \eqref{eq:T1-bound}, to handle the case where
$k$-orthogonality \eqref{eq:orthprop} holds or when $s_\ell = s_{\ell-1}$.

These together with some further estimations lead to the following
simplified mean-square error bound.
\begin{theorem} \label{thm:error-simp}
Under Assumptions \textnormal{(\textbf{A1})}--\textnormal{(\textbf{A6})}
and the condition \eqref{eq:nasty} with $n=3$, for every $f\in
H^{-1+t}(D)$ with $0\le t\le 1$ and every $G\in H^{-1+t'}(D)$ with $0 \le
t'\le 1$, the mean-square error of the multi-level QMC FE algorithm
defined by \eqref{eq:MLQMCFE} can be estimated as follows
\begin{align} \label{eq:error-simp}
  &\bbE[|I(G(u)) - Q_*^L(\cdot;G(u))|^2]
  \;\le\; C\, D_{\bsgamma}(\lambda)\,
  \|f\|_{H^{-1+t}(D)}^2\,\|G\|_{H^{-1+t'}(D)}^2 \nonumber\\
  &\;\cdot
  \left[
  \left( h_L^{\tau} + \theta_L\, s_L^{-2(1/p-1)} \right)^2
  + \sum_{\ell=0}^L [\varphi(N_\ell)]^{-1/\lambda}
    \left(h_{\ell-1}^{\tau} + \theta_{\ell-1}\,s_{\ell-1}^{-\min(1/p-1,\alpha)}\right)^2
    \right]\;,
\end{align}
where
\begin{align} \label{eq:Dgamma}
  D_\bsgamma(\lambda)
  \,:=\,
  \Bigg(\sum_{|\setu|<\infty}
  \gamma_\setu^\lambda\, [\rho(\lambda)]^{|\setu|}\Bigg)^{1/\lambda}\,
  \Bigg(\sum_{|\setu|<\infty}
  \frac{[(|\setu|+3)!]^2\prod_{j\in\setu} \bar{b}_j^2}{\gamma_\setu}
  \Bigg)\;,
\end{align}
with $0\le\tau:=t + t'\le 2$, $h_{-1} :=1$, $s_{-1} := 1$, $\theta_{-1} :=
0$, $\rho(\lambda)$ as in \eqref{eq:rho}, and $\bar{b}_j$ as in
\eqref{eq:defbarbj}.
In general we have $\theta_\ell = 1$ for all $\ell=0,\ldots,L$. If $s_\ell
= s_{\ell-1}$ for some $\ell\ge 1$ then $\theta_{\ell-1} = 0$. When
$k$-orthogonality \eqref{eq:orthprop} holds we have $\theta_\ell = 0$ for
all $\ell=0,\ldots,L$. Assumptions \textnormal{(\textbf{A3})} and
\textnormal{(\textbf{A5})} and the condition \eqref{eq:nasty} are not
required when $\theta_\ell= 0$ for all $\ell$.
The expectation $\bbE[\cdot]$ is with respect to the random compound shift
which is drawn from the uniform distribution over $[0,1]^{s_*}$.
The error bound \eqref{eq:error-simp} is meaningful only if
$D_\bsgamma(\lambda)$ is finite.
\end{theorem}
\subsection{Choosing the Parameter $\lambda$ and the Weights $\gamma_\setu$}
\label{ssec:weights}
Following \cite{KSS1}, we now choose the weights $\gamma_\setu$ to
minimize $D_\bsgamma(\lambda)$. We also specify the value of $\lambda$ to
get the best convergence rate possible. Note that our goal is to have
$\lambda$ as small as possible, since a smaller value of $\lambda$ yields
a better convergence rate with respect to the number of QMC points.

In the following theorem, the assumption \eqref{eq:assump-1} is implied by
Assumption (\textbf{A7}).
\begin{theorem} \label{thm:weights}
With $\bar{b}_j$ defined as in \eqref{eq:defbarbj} for fixed
$\kappa\in (0,1]$, suppose that
\begin{equation} \label{eq:assump-1}
  \sum_{j\ge 1} \bar{b}_j^q \,<\, \infty
  \qquad\mbox{for some}\quad
  0 < q \le 1\;,
\end{equation}
and when $q=1$ assume additionally that
\begin{equation} \label{eq:small-1}
  \sum_{j\ge 1} \bar{b}_j \,<\, \sqrt{6}\;.
\end{equation}
For a given $\lambda\in(1/2,1]$, the choice of weights
\begin{equation} \label{eq:choiceweight-1}
  \gamma_\setu \,=\,
  \gamma_\setu^*(\lambda) \,:=\,
  \Bigg(\frac{(|\setu|+3)!}{6}\,
  \prod_{j\in\setu} \frac{\bar{b}_j}{\sqrt{\rho(\lambda)}}
  \Bigg)^{2/(1+\lambda)}
\end{equation}
minimizes $D_\bsgamma(\lambda)$ given in \eqref{eq:Dgamma}, if
$D_{\bsgamma^*}(\lambda)<\infty$. Moreover, the choice of $\lambda$ given
by
\begin{equation} \label{eq:def_lam}
  \lambda \,=\, \lambda_q \,:=\,
  \begin{cases}
  \displaystyle\frac{1}{2-2\delta} \quad\mbox{for some}\quad
  \delta\in (0,1/2) & \mbox{when } q\in (0,2/3]\;,
  \\
  \displaystyle\frac{q}{2-q} & \mbox{when } q\in (2/3,1)\;,
  \\
  1 & \mbox{when } q = 1\;,
  \end{cases}
\end{equation}
together with $\gamma_\setu=\gamma_\setu^*(\lambda_q)$, ensures that
$D_{\bsgamma^*}(\lambda_q) < \infty$, and thus justifies the error
bound \eqref{eq:error-simp}.
%
%
\end{theorem}
%
%
\begin{proof}
This proof follows closely the proof of \cite[Theorem~6.4]{KSS1}. Apart
from the simple replacement of $b_j$ by $\bar{b}_j$ and of $p$ by $q$, the
main difference is that we now have to handle a sum containing the factor
$(|\setu|+3)!$ instead of $|\setu|!$. For this we make use of
\cite[Lemma~6.3]{KSS1} with $n=3$ instead of $n=0$.

Using \cite[Lemma~6.2]{KSS1}, we see that $D_\bsgamma(\lambda)$ is
minimized by choosing $\gamma_\setu$ as in \eqref{eq:choiceweight-1} for
$|\setu|<\infty$, provided that $D_\bsgamma(\lambda) < \infty$. We add
that an overall rescaling of weights does not affect the minimization
argument. Our choice of scaling here is consistent with the convention
that $\gamma_\emptyset := 1$.

In the course of our derivation below we eventually choose the value of
$\lambda$ depending on the value of $q$, but until then $\lambda$ and $q$
will be independent.
For the weights given by \eqref{eq:choiceweight-1}, we have
\begin{align*}
  \sum_{|\setu|<\infty} (\gamma_\setu^*)^\lambda\, [\rho(\lambda)]^{|\setu|}
  \,=\, 6^{-2\lambda/(1+\lambda)}\, A_\lambda\;,
  \quad
  \sum_{|\setu|<\infty} \frac{[(|\setu|+3)!]^2 \prod_{j\in\setu} \bar{b}_j^2}{\gamma_\setu^*}
  &\,=\, 6^{2/(1+\lambda)} A_\lambda\;,
  \end{align*}
and thus $D_{\bsgamma^*}(\lambda) = A_\lambda^{1/\lambda+1}$, where
\[
  A_\lambda \,:=\,
  \sum_{|\setu|<\infty} [(|\setu|+3)!]^{2\lambda/(1+\lambda)}
  \prod_{j\in\setu} \left(\bar{b}_j^{2\lambda}\rho(\lambda)\right)^{1/(1+\lambda)}\;.
\]

For $\lambda\in (1/2,1)$, we have $2\lambda/(1+\lambda)<1$ and we
further estimate $A_\lambda$ as follows:
we multiply and divide each term in the expression
by $\prod_{j\in\setu} \alpha_j^{2\lambda/(1+\lambda)}$, with $\alpha_j>0$
to be specified later, and then apply H\"older's inequality with
conjugate exponents $(1+\lambda)/(2\lambda)$ and
$(1+\lambda)/(1-\lambda)$, to obtain
\begin{align*}
  A_\lambda
  &\,=\, \sum_{|\setu|<\infty} [(|\setu|+3)!]^{2\lambda/(1+\lambda)}
  \prod_{j\in\setu} \alpha_j^{2\lambda/(1+\lambda)}
  \prod_{j\in\setu} \left(\frac{\bar{b}_j^{2\lambda}\rho(\lambda)}{\alpha_j^{2\lambda}}\right)^{1/(1+\lambda)} \\
  &\,\le\, \left(\sum_{|\setu|<\infty} (|\setu|+3)!
  \prod_{j\in\setu} \alpha_j \right)^{2\lambda/(1+\lambda)}
  \left(\sum_{|\setu|<\infty}
  \prod_{j\in\setu} \left(\frac{\bar{b}_j^{2\lambda}\rho(\lambda)}{\alpha_j^{2\lambda}}\right)^{1/(1-\lambda)}
  \right)^{(1-\lambda)/(1+\lambda)} \\
  &\,\le\,
  \left[6\left(\frac{1}{1-\sum_{j\ge 1}\alpha_j}\right)^{4}\right]^{2\lambda/(1+\lambda)}
  \exp\left(\frac{1-\lambda}{1+\lambda}\, [\rho(\lambda)]^{1/(1-\lambda)}
  \sum_{j\ge 1} \left(\frac{\bar{b}_j}{\alpha_j}\right)^{2\lambda/(1-\lambda)}
  \right)\;
\end{align*}
which holds and $A_\lambda$ is finite, see \cite[Lemma~6.3]{KSS1},
provided that
\begin{equation} \label{eq:cond_rj}
  \sum_{j\ge 1} \alpha_j \,<\, 1
  \qquad\mbox{and}\qquad
  \sum_{j\ge 1} \left(\frac{\bar{b}_j}{\alpha_j}\right)^{2\lambda/(1-\lambda)} \,<\, \infty\;.
\end{equation}
We now choose
\begin{equation*} 
  \alpha_j \,:=\, \frac{\bar{b}_j^q}{\varpi}
  \qquad\mbox{for some parameter}\quad \varpi > \sum_{j\ge1} \bar{b}_j^q
\;.
\end{equation*}
Then the first sum in \eqref{eq:cond_rj} is less than $1$ due to the
assumption \eqref{eq:assump-1}. Noting that \eqref{eq:assump-1}
implies that $\sum_{j\ge1} \bar{b}_j^{q'}<\infty$ for all $q'\ge q$,
we conclude that the second sum in \eqref{eq:cond_rj} converges for
\[
  \frac{2\lambda}{1-\lambda} (1-q) \,\ge\, q
  \qquad\iff\qquad
  q \,\le\, \frac{2\lambda}{1+\lambda}
  \qquad\iff\qquad
  \lambda \,\ge\,\frac{q}{2-q}\;.
\]
Since $\lambda$ must be strictly between $1/2$ and $1$, when $q\in(0,2/3]$
we choose $\lambda_q = 1/(2-2\delta)$ for some $\delta\in (0,1/2)$, and
when $q \in(2/3,1)$ we set $\lambda_q = q/(2-q)$.

For the case $q=1$ we take $\lambda_q = 1$, and we use $\rho(1)=1/6$. Then
using \cite[Lemma~6.3]{KSS1} and the assumption \eqref{eq:small-1} we
obtain
\[
  A_1
  \,=\, \sum_{|\setu|<\infty} (|\setu|+3)! \prod_{j\in\setu} \left(\frac{\bar{b}_j}{\sqrt{6}}\right)
  \,\le\, 6\left(\frac{1}{1-\sum_{j\ge 1} (\bar{b}_j/\sqrt{6})}\right)^{4}
  \,<\, \infty\;.
\]
This completes the proof. \quad\qed
\end{proof}

In the following theorem we verify that with a slightly modified choice of
weights the condition \eqref{eq:nasty}, which is required in
Theorem~\ref{thm:error-simp}, is indeed satisfied. The assumptions in the
theorem are consistent with Assumptions (\textbf{A3}), (\textbf{A5}), and
(\textbf{A7}), however, the requirement that $p$ be strictly smaller than
$q$ is new, and is essential for obtaining the decay we need.
%
\begin{theorem} \label{thm:trunc-decay}
With $b_j$ and $\bar{b}_j$ defined as in \eqref{eq:defbj} and
\eqref{eq:defbarbj} for fixed $\kappa\in (0,1]$, suppose that the sequence
$\{b_j\}$ is non-increasing and
\[
  \sum_{j\ge 1} b_j^p \,<\, \infty
  \qquad\mbox{and}\qquad
  \sum_{j\ge 1} \bar{b}_j^q \,<\, \infty
  \qquad\mbox{for some}\quad 0 < p < q \le 1\;.
\]
Define a new sequence $\{\beta_j\}$ by
\begin{equation} \label{eq:beta}
  \beta_j \,:=\, \max(\bar{b}_j, b_j^{p/q})\;.
\end{equation}
Then, Theorems~\ref{thm:error-simp} and~\ref{thm:weights} hold \emph{a
fortiori} if $\bar{b}_j$ is replaced by $\beta_j$. Moreover, the choice of
weights \eqref{eq:choiceweight-1} with $\beta_j$ instead of $\bar{b}_j$
satisfies the condition \eqref{eq:nasty} with $n=3$ and
\begin{equation} \label{eq:exponent}
  \alpha \,=\, \frac{1}{p} - \frac{1}{q}\;.
\end{equation}
However, the constant $C$ in Theorem \ref{thm:error-simp} has now a
dependence on $\lambda$.
\end{theorem}

\begin{proof}
Note that $\sum_{j\ge 1} \beta_j^q < \infty$. Substituting
\eqref{eq:choiceweight-1}, with $\bar{b}_j$ replaced by $\beta_j$, into
the left-hand side of \eqref{eq:nasty}, we obtain
\begin{align*}
  &\sum_{\satop{\setu\subseteq\{1:s_\ell\}}{\setu\cap\{s_{\ell-1}+1:s_\ell\}\ne\emptyset}}
  \!\!\!\!\!\!
  \frac{(|\setu|!)^2\prod_{j\in\setu} b_j^2}{\gamma_\setu}
  \,=\,
  \sum_{\satop{\setu\subseteq\{1:s_\ell\}}{\setu\cap\{s_{\ell-1}+1:s_\ell\}\ne\emptyset}}
  \frac{(|\setu|!)^2\prod_{j\in\setu} b_j^2}
       {[\frac{1}{6}(|\setu|+3)! \prod_{j\in\setu} (\beta_j/\!\!\sqrt{\rho(\lambda)})]^{2/(1+\lambda)}} \\
  &
  \,\le\, \sum_{k=s_{\ell-1}+1}^{s_\ell} \sum_{k\in\setu\subseteq\{1:s_\ell\}}
  \frac{(|\setu|!)^2\prod_{j\in\setu} b_j^2}
       {[\frac{1}{6}(|\setu|+3)! \prod_{j\in\setu} (\beta_j/\!\!\sqrt{\rho(\lambda)})]^{2/(1+\lambda)}} \\
  &\,=\, \sum_{k=s_{\ell-1}+1}^{s_\ell} \sum_{\setv\subseteq\{1:s_\ell\}\setminus\{k\}}
  \frac{b_k^2}{[\beta_k/\!\!\sqrt{\rho(\lambda)}]^{2/(1+\lambda)}}
  \frac{[(|\setv|+1)!]^2\prod_{j\in\setv} b_j^2}
       {[\frac{1}{6}(|\setv|+1+3)! \prod_{j\in\setv} (\beta_j/\!\!\sqrt{\rho(\lambda)})]^{2/(1+\lambda)}} \\
  &\,\le\, [\rho(\lambda)]^{1/(1+\lambda)}
  \sum_{k=s_{\ell-1}+1}^{s_\ell} b_k^{2-2(p/q)/(1+\lambda)}
  \sum_{\setv\subseteq\{1:s_\ell\}}
  \frac{[(|\setv|+3)!]^2\prod_{j\in\setv} b_j^2}{\gamma_\setv}\;,
\end{align*}
where in the last step we allowed $\setv$ to also include the index $k$,
and used $\beta_k\ge b_k^{p/q}$ and $(|\setv|+1+3)! \ge (|\setv|+3)!$ in
the denominator, and $(|\setv|+1)! \le (|\setv|+3)!$ in the numerator.

To complete the proof, we estimate the tail sum $\sum_{k\ge s_{\ell-1}+1}
b_k^{2-2(p/q)/(1+\lambda)}$ using \eqref{eq:stechkin}, but with $b_j$
replaced by $b_j^{2-2(p/q)/(1+\lambda)}$ and $p$ replaced by
$p/[2-2(p/q)/(1+\lambda)]$. This is valid because
\[
   \frac{2-2(p/q)/(1+\lambda)}{p}
   \,\ge\, \frac{2}{p} - \frac{2}{q (1 + q/(2-q))}
   \,=\, \frac{2}{p}- \frac{2}{q}\,+\,1\,>\,1\;,
\]
where we used $\lambda\ge q/(2-q)$. The exponent of $s_{\ell-1}$ in
\eqref{eq:nasty} becomes $-(2/p-2/q)$, proving that \eqref{eq:nasty} holds
with $\alpha = 1/p-1/q$, but with a constant in front that now depends on
$\lambda$. \quad\qed
\end{proof}

\subsection{Summary of Overall Cost Versus Error}
\label{ssec:SumCostErr}

Recall that
\begin{equation} \label{eq:def_h}
  h_\ell \,\asymp\, 2^{-\ell}
  \quad\mbox{and}\quad
  M_{h_\ell} \,\asymp\, h_\ell^{-d} \,\asymp\, 2^{\ell d}
  \quad\mbox{for}\quad \ell=0,\ldots,L\;.
\end{equation}
Based on the mean square error bound \eqref{eq:error-simp}, we now specify
$s_\ell$ and $N_\ell$ for each level. We consider two scenarios depending
on whether or not $k$-orthogonality \eqref{eq:orthprop} holds.

For our cost model we assume the availability of a linear complexity FE
solver. We assume that in general the cost for assembling the stiffness
matrix at level $\ell$ is $\calO(s_\ell\,M_{h_\ell})$, and is
$\calO(M_{h_\ell}\,\log(M_{h_\ell}))$ if $k$-orthogonality
\eqref{eq:orthprop} holds (see the second part of
Theorem~\ref{thm:orthprop}). Moreover, we assume that the functions
$\psi_j$ are explicitly known, and that integration of any basis functions
in the FE method against any $\psi_j$ is available at unit cost. Thus
\[
  {\rm cost} \,=\, \calO\left(\sum_{\ell=0}^L N_\ell\, K_\ell\right)\;,
  \quad
  K_\ell \,:=\,
  \begin{cases}
  h_\ell^{-d}\,\log(h_\ell^{-d})
  & \mbox{if $k$-orthogonality \eqref{eq:orthprop} holds}\;, \\
  h_\ell^{-d}\,s_\ell & \mbox{otherwise}\;.
  \end{cases}
\]
Clearly, changing the cost model may change the definition of $K_\ell$.
(Some cost models in the literature do not include $s_\ell$ as part of
$K_\ell$.) Note that our cost model does not include the pre-computation
cost for the CBC construction of randomly shifted lattice rules, which
requires $\calO(s_\ell\,N_\ell\,\log N_\ell + s_\ell^2\,N_\ell)$
operations on level $\ell$.

\smallskip\noindent\emph{Scenario 1.} In the special case where
$k$-orthogonality \eqref{eq:orthprop} holds, the values of $s_\ell$ are
given by \eqref{eq:def_s_k}, and we have $\theta_\ell = 0$ for all $\ell$
in the error bound \eqref{eq:error-simp}, giving the mean square error
bound (denoted in this subsection by error$^2$ for simplicity)
\begin{equation} \label{eq:error-short}
  {\rm error}^2 \,=\,
  \calO \left(
  h_L^{2\tau}
  + \sum_{\ell=0}^L [\varphi(N_\ell)]^{-1/\lambda} h_{\ell-1}^{2\tau} \right)\;.
\end{equation}
\smallskip\noindent\emph{Scenario 2.} When $k$-orthogonality
\eqref{eq:orthprop} does not hold and $p<q\le 1$, we have $\theta_L = 1$
in the error bound \eqref{eq:error-simp}. We assume that the weights
$\gamma_\setu$ are chosen as in Theorem \ref{thm:trunc-decay}, so that
\eqref{eq:exponent} holds. To balance the error contribution within the
highest discretization level, we impose the condition $s_L^{-2(1/p-1)} =
\calO( h_L^{\tau} )$, which is equivalent to $s_L = \Omega( 2^{L\tau
p/(2-2p)})$. Then, to minimize the error within each level, one choice for
$s_\ell$ is to set $s_\ell = s_L$ for all $\ell<L$, leading to
$\theta_{\ell-1} = 0$ for all $\ell=1,\ldots,L$ in \eqref{eq:error-simp}.
Alternatively, since $s_\ell$ should be as small as possible from the
point of view of reducing the cost at each level, we can impose the
condition $s_{\ell-1}^{-(1/p-1/q)} = \calO(h_{\ell-1}^{\tau} )$ for
$\ell=1,\ldots,L$ (see \eqref{eq:error-simp} with $\alpha=1/p-1/q$), which
is equivalent to $s_\ell = \Omega( 2^{\ell\tau pq/(q-p)} )$ for
$\ell=0,\ldots,L-1$. Combining both approaches, while taking into account
the monotonicity condition \eqref{eq:sellincr}, we choose
\begin{align} \label{eq:def_s}
  s_\ell \,:=\,
  \min \Big( \big\lceil 2^{\ell\tau pq/(q-p)}\big\rceil ,
  \big\lceil 2^{L\tau p/(2-2p)}\big\rceil \Big)
  \quad\mbox{for}\quad \ell=0,\ldots,L\;.
\end{align}
Thus we have $s_\ell$ strictly increasing for $\ell=0,\ldots, \lfloor
L(q-p)/(q(2-2p))\rfloor$,
and the remaining $s_\ell$ are all identical. 
This leads again to the error bound \eqref{eq:error-short}.

\smallskip\noindent\emph{Scenario 3.} When $k$-orthogonality
\eqref{eq:orthprop} does not hold and $p=q< 1$, we choose
\begin{equation} \label{eq:s3}
  s_\ell \,:=\, \big\lceil 2^{L\tau p/(2-2p)}\big\rceil
  \quad\mbox{for}\quad \ell=0,\ldots,L\;.
\end{equation}
This again yields the error bound \eqref{eq:error-short}.

\smallskip
We remark that for all $N\in \mathbb{N}$, the Euler totient function
$\varphi(N)$ takes values close to $N$. Specifically, if $N$ is prime then
$1/\varphi(N) = 1/(N-1)\le 2/N$. If $N$ is a power of $2$ then
$1/\varphi(N) = 2/N$. It is known from \cite[Theorem 8.8.7]{BS66} that
$1/\varphi(N) < (e^\Upsilon\log\log N + 3/\log\log N)/N$ for all $N\ge 3$,
where $e^\Upsilon = 1.781\ldots$. Thus it can be verified that for all
computationally realistic values of $N$, say, $N\le 10^{30}$, we have
$1/\varphi(N) < 9/N$.
Treating this factor $9$ as a constant and using $h_{\ell-1} \asymp
h_{\ell}$, we obtain for all three scenarios the simpler mean square error
expression
\[
  {\rm error}^2 \,=\,
  \calO \left(
  h_L^{2\tau}
  + \sum_{\ell=0}^L N_\ell^{-1/\lambda} h_{\ell}^{2\tau} \right)\;.
\]

To \emph{minimize the mean square error for a fixed cost}, we consider the
Lagrange multiplier function
\[
  g(\mu) \,:=\, \underbrace{h_L^{2\tau}
  + \sum_{\ell=0}^L N_\ell^{-1/\lambda} h_{\ell}^{2\tau}}_{\mbox{\footnotesize{mean square error}}}
  \;+\; \mu\; \underbrace{\sum_{\ell=0}^L N_\ell\, K_\ell}_{\mbox{\footnotesize{cost}}} \;.
\]
We look for the stationary point of $g(\mu)$ with respect to $N_\ell$,
thus demanding that
\[
  \frac{\partial g(\mu)}{\partial N_\ell}
  \,=\,
  -\frac{1}{\lambda} N_\ell^{-1/\lambda-1} h_{\ell}^{2\tau} + \mu\,K_\ell \,=\, 0
  \qquad\mbox{for}\quad \ell=0,\ldots,L\;.
 \]
This prompts us to define
\begin{equation} \label{eq:def_N}
  N_\ell
  \,:=\,
  \Big\lceil N_0
  \left(h_0^{-2\tau}\,K_0\, h_\ell^{2\tau}\,K_\ell^{-1} \right)^{\lambda/(\lambda+1)}
  \Big\rceil
  \qquad\mbox{for}\quad \ell=1,\ldots,L\;.
\end{equation}
Leaving $N_0$ to be specified later and treating $h_0$ and $K_0$ as
constants, we conclude that
\begin{align} \label{eq:error_cost}
  {\rm error}^2 \,=\, \calO \left(
  h_L^{2\tau}
  \;+\; N_0^{-1/\lambda}
  \sum_{\ell=0}^L E_\ell \right)
  \quad\mbox{and}\quad
  {\rm cost} \,=\, \calO \left(
  N_0\,
  \sum_{\ell=0}^L E_\ell \right)\;,
\end{align}
where
\begin{align*}
  E_\ell
  := (h_\ell^{2\lambda\tau}\,K_\ell)^{1/(\lambda+1)}
  =\!
  \begin{cases}
  (h_\ell^{2\lambda\tau-d}\,\log(h_\ell^{-d}))^{1/(\lambda+1)}
  & \mbox{if $k$-orthogonality \eqref{eq:orthprop} holds}\,, \\
  (h_\ell^{2\lambda\tau-d}\,s_\ell)^{1/(\lambda+1)} & \mbox{otherwise}\;.
  \end{cases}
\end{align*}
We see that the mean square error is \emph{not} necessarily minimized by
balancing the error terms between the levels. For example, when
$k$-orthogonality \eqref{eq:orthprop} holds, we observe that
\begin{itemize}
\item For $d < 2\lambda\tau$, the quantity $E_\ell$ (and thus the mean
    square error and cost at level~$\ell$) decreases with
    increasing~$\ell$.
\item For $d > 2\lambda\tau$, the quantity $E_\ell$ increases with
    increasing $\ell$.
\end{itemize}

In the light of the error bound in \eqref{eq:error_cost}, we always choose
$N_0$ to satisfy
\begin{equation} \label{eq:cond_N0}
  N_0^{-1/\lambda}\, \sum_{\ell=0}^L E_\ell \,=\, \calO( h_L^{2\tau} )
  \quad\iff\quad
  N_0 \,=\, \Omega \bigg( h_L^{-2\tau\lambda} \bigg(\sum_{\ell=0}^L E_\ell\bigg)^\lambda\bigg)\;,
\end{equation}
leading to the simplified error bound ${\rm error}^2 = \calO \big(
h_L^{2\tau} \big)$.

\smallskip\noindent\emph{Scenario 1 (continued).}
Substituting $h_\ell\asymp 2^{-\ell}$, we obtain for the case where
$k$-orthogonality holds that
\begin{align*}
  \sum_{\ell=0}^L E_\ell
  &\,=\,
  \calO\left(
  \sum_{\ell=0}^L 2^{-\ell(2\lambda\tau-d)/(\lambda+1)} (\ell+1)^{1/(\lambda+1)}
  \right) \\
  &\,=\,
  \begin{cases}
  \calO\big( 1 \big)
  & \mbox{if } d < 2\lambda\tau\,,
  \\
  \calO\big( L^{(\lambda+2)/(\lambda+1)}\big)
  & \mbox{if } d = 2\lambda\tau\,,
  \\
  \calO\big(
  2^{-L(2\lambda\tau-d)/(\lambda+1)} L^{1/(\lambda+1)}
  \big)
  & \mbox{if } d > 2\lambda\tau \,.
  \end{cases}
\end{align*}
The choice \eqref{eq:cond_N0} for $N_0$ then yields
\begin{align}
\label{eq:def_N0_k}
  N_0
  &\,:=\,
  \begin{cases}
  \big\lceil 2^{L\tau (2\lambda)} \big\rceil
  & \mbox{if } d < 2\tau\lambda\,,
  \\
  \big\lceil 2^{L\tau (2\lambda)} L^{\lambda(\lambda+2)/(\lambda+1)} \big\rceil
  & \mbox{if } d = 2\tau\lambda\,,
  \\
  \big\lceil 2^{L\tau (d/\tau+2)\lambda/(\lambda+1)} L^{\lambda/(\lambda+1)} \big\rceil
  & \mbox{if } d > 2\tau\lambda \,.
  \end{cases}
\end{align}
Upon substituting \eqref{eq:cond_N0} into the cost bound in
\eqref{eq:error_cost} and using \eqref{eq:def_N0_k}, we obtain
\begin{align*}
  {\rm cost}
  \,=\, \calO \big(N_0^{(\lambda+1)/\lambda} h_L^{2\tau}\big)
  \,=\,
  \begin{cases}
  \calO\big( 2^{L\tau (2\lambda)} \big)
  & \mbox{if } d < 2\lambda\tau\,,
  \\
  \calO\big( 2^{L\tau (2\lambda)}L^{\lambda+2} \big)
  & \mbox{if } d = 2\lambda\tau\,,
  \\
  \calO\big(
  2^{L\tau(d/\tau)} L
  \big)
  & \mbox{if } d > 2\lambda\tau \,.
  \end{cases}
\end{align*}

\smallskip\noindent\emph{Scenario 2 (continued).}
When $k$-orthogonality does not hold and $p<q\le 1$, we use the definition
\eqref{eq:def_s} for $s_\ell$. We consider separately the two alternative
choices in \eqref{eq:def_s}: choice A takes $s_\ell = \lceil
2^{\ell\tau\eta} \rceil$ for all $\ell$, while choice B takes $s_\ell =
\lceil 2^{L\tau\xi} \rceil$ for all $\ell$, where for ease of notation we
have introduced
\begin{equation} \label{eq:etakappa}
  \eta \,:=\, \frac{pq}{q-p} \qquad\mbox{and}\qquad \xi \,:=\, \frac{p}{2-2p}\;,
\end {equation}
noting that $\eta \ge \xi$. Then we have $\sum_{\ell=0}^L E_\ell \le
\min (\sum_{\ell=0}^L E_\ell^{(A)},\sum_{\ell=0}^L E_\ell^{(B)})$,
%
where
\begin{align} \label{eq:EA}
  \sum_{\ell=0}^L E_\ell^{(A)}
  &\,=\, \calO \Bigg( \sum_{\ell=0}^L 2^{\ell\tau(d/\tau-2\lambda + \eta)/(\lambda+1)} \Bigg) \nonumber\\
  &\,=\,
  \begin{cases}
  \calO \big(1 \big)
    & \mbox{if } d/\tau < 2\lambda-\eta\;, \\
  \calO \big(L\big)
    & \mbox{if } d/\tau = 2\lambda-\eta\;, \\
  \calO \big(2^{L\tau (d/\tau-2\lambda + \eta)/(\lambda+1)} \big)
    & \mbox{if } d/\tau > 2\lambda-\eta\;,
  \end{cases}
\end{align}
\begin{align}
  \label{eq:EB}
  \sum_{\ell=0}^L E_\ell^{(B)}
  &\,=\, \calO \Bigg( 2^{L\tau\xi/(\lambda+1)} \sum_{\ell=0}^L 2^{\ell\tau (d/\tau-2\lambda)/(\lambda+1)} \Bigg) \nonumber\\
  &\,=\,
  \begin{cases}
  \calO \big(2^{L\tau\xi/(\lambda+1)} \big)
    & \mbox{if } d/\tau < 2\lambda \;, \\
  \calO \big(2^{L\tau\xi/(\lambda+1)} L \big)
    & \mbox{if } d/\tau = 2\lambda\;, \\
  \calO \big(2^{L\tau (d/\tau-2\lambda + \xi)/(\lambda+1)} \big)
    & \mbox{if } d/\tau > 2\lambda\;.
  \end{cases}
\end{align}
For the ``middle case'' $2\lambda - \eta < d/\tau < 2\lambda$, it is
beneficial to estimate directly
\begin{align*}
  \sum_{\ell=0}^L E_\ell
  &\,=\, \calO \Bigg( \sum_{\ell=0}^{\lfloor L\xi/\eta\rfloor} 2^{\ell\tau(d/\tau - 2\lambda + \eta)/(\lambda+1)}
  + 2^{L\tau\xi/(\lambda+1)} \sum_{\ell=\lfloor L\xi/\eta\rfloor+1}^L
  2^{\ell\tau(d/\tau - 2\lambda)/(\lambda+1)} \Bigg) \\
  &\,=\,
  \calO\big(
  2^{L\tau(\xi/\eta)(d/\tau - 2\lambda + \eta)/(\lambda+1)} \big) \;.
\end{align*}
Comparing this with \eqref{eq:EA} and \eqref{eq:EB}, and taking the
appropriate minimum, we obtain
\begin{align*}
  \sum_{\ell=0}^L E_\ell
  &\,=\,
  \begin{cases}
  \calO \big(1 \big)
    & \mbox{if } d/\tau < 2\lambda-\eta\;, \\
  \calO \big(L\big)
    & \mbox{if } d/\tau = 2\lambda-\eta\;, \\
  \calO \big(2^{L\tau (\xi/\eta)(d/\tau-2\lambda + \eta)/(\lambda+1)} \big)
    & \mbox{if } 2\lambda-\eta < d/\tau < 2\lambda \;, \\
  \calO \big(2^{L\tau\xi/(\lambda+1)} L \big)
    & \mbox{if } d/\tau = 2\lambda\;, \\
  \calO \big(2^{L\tau (d/\tau-2\lambda + \xi)/(\lambda+1)} \big)
    & \mbox{if } d/\tau > 2\lambda\;.
  \end{cases}
\end{align*}
The choice \eqref{eq:cond_N0} for $N_0$ yields
\begin{align} \label{eq:def_N0}
  N_0
  &\,:=\,
  \begin{cases}
  \lceil 2^{L\tau(2\lambda)} \rceil
    & \mbox{if } d/\tau < 2\lambda - \eta\;, \\
  \lceil 2^{L\tau(2\lambda)}\,L^\lambda \big)
    & \mbox{if } d/\tau = 2\lambda - \eta\;, \\
  \lceil 2^{L\tau [2(\lambda+1) + (\xi/\eta)(d/\tau - 2\lambda+\eta)]\lambda/(\lambda+1)} \rceil
    & \mbox{if } 2\lambda - \eta < d/\tau < 2\lambda\;, \\
  \lceil 2^{L\tau [2(\lambda+1) +\xi]\lambda/(\lambda+1)} L^\lambda \rceil
    & \mbox{if } d/\tau = 2\lambda\;, \\
  \lceil 2^{L\tau [2 + d/\tau + \xi]\lambda/(\lambda+1)} \rceil
    & \mbox{if } d/\tau > 2\lambda\;.
  \end{cases}
\end{align}
Then we have ${\rm error}^2 = \calO ( h_L^{2\tau})$ as before, but now
\begin{align*}
  {\rm cost} &\,=\, \calO \big(
  N_0^{(\lambda+1)/\lambda} h_L^{2\tau}\big)
  \,=\,
  \begin{cases}
  \calO \big(2^{L\tau(2\lambda)} \big)
    & \mbox{if } d/\tau < 2\lambda - \eta\;, \\
  \calO \big(2^{L\tau(2\lambda)}\,L^{\lambda+1} \big)
    & \mbox{if } d/\tau = 2\lambda - \eta\;, \\
  \calO \big(2^{L\tau [2\lambda + (\xi/\eta)(d/\tau - 2\lambda+\eta)]} \big)
    & \mbox{if } 2\lambda -\eta < d/\tau < 2\lambda \;, \\
  \calO \big(2^{L\tau (2\lambda + \xi)} L^{\lambda+1} \big)
    & \mbox{if } d/\tau = 2\lambda\;, \\
  \calO \big(2^{L\tau (d/\tau +\xi)} \big)
    & \mbox{if } d/\tau > 2\lambda\;.
  \end{cases}
\end{align*}

\smallskip\noindent\emph{Scenario 3 (continued). When
$k$-orthogonality \eqref{eq:orthprop} does not hold and $p=q< 1$, we
proceed in a similar way, taking $s_\ell=\lceil2^{2L\tau\xi\rceil}$ with
$\xi$ given by \eqref{eq:etakappa}, to obtain
\begin{align} \label{eq:def_N03}
  N_0
  &\,:=\,
  \begin{cases}
  \lceil 2^{L\tau [2(\lambda+1) + \xi]\lambda/(\lambda+1)} \rceil
    & \mbox{if } d/\tau < 2\lambda\;, \\
  \lceil 2^{L\tau [2(\lambda+1) +\xi]\lambda/(\lambda+1)} L^\lambda \rceil
    & \mbox{if } d/\tau = 2\lambda\;, \\
  \lceil 2^{L\tau [2 + d/\tau + \xi]\lambda/(\lambda+1)} \rceil
    & \mbox{if } d/\tau > 2\lambda\;,
  \end{cases}
\end{align}
and
\begin{align*}
  {\rm cost} \,=\,
  \begin{cases}
  \calO \big(2^{L\tau (2\lambda + \xi)} \big)
    & \mbox{if } d/\tau < 2\lambda \;, \\
  \calO \big(2^{L\tau (2\lambda + \xi)} L^{\lambda+1} \big)
    & \mbox{if } d/\tau = 2\lambda\;, \\
  \calO \big(2^{L\tau (d/\tau +\xi)} \big)
    & \mbox{if } d/\tau > 2\lambda\;.
  \end{cases}
\end{align*}}

\smallskip
In all three scenarios, for given $\varepsilon > 0$, we choose $L$ such
that
\begin{equation} \label{eq:def_L}
  h_L^\tau \,\asymp\, 2^{-L\tau} \,\asymp\, \varepsilon\;.
\end{equation}
We can then express the total cost of the algorithm in terms of
$\varepsilon$. This is summarized in Theorem~\ref{thm:summary} below.
\begin{theorem} \label{thm:summary}
Under Assumptions \textnormal{(\textbf{A1})}--\textnormal{(\textbf{A7})},
leaving out \textnormal{(\textbf{A5})} if $k$-orthogonality
\eqref{eq:orthprop} holds, for $f\in H^{-1+t}(D)$ and $G\in H^{-1+t'}(D)$
with $0 \le t,t'\le 1$ and $\tau:=t+t'>0$, consider the multi-level QMC FE
algorithm defined by \eqref{eq:MLQMCFE}. Given $\varepsilon>0$, with $L$
given by \eqref{eq:def_L}, $h_\ell$ given by \eqref{eq:def_h}, $s_\ell$
given by \eqref{eq:def_s_k}, \eqref{eq:def_s} or \eqref{eq:s3} as
appropriate, $N_\ell$ given by \eqref{eq:def_N}, $N_0$ given by
\eqref{eq:def_N0_k}, \eqref{eq:def_N0} or \eqref{eq:def_N03} as
appropriate, and with randomly shifted lattice rules constructed based on
POD weights $\gamma_\setu$ given by \eqref{eq:choiceweight-1}, in which
$\bar{b}_j$ is replaced by $\beta_j$ from \eqref{eq:beta},
we obtain
\[
  \sqrt{\bbE[|I(G(u)) - Q_*^L(\cdot;G(u))|^2]}
  \,=\, \calO \left(\varepsilon \right)
  \;,
\]
and
\[
  {\rm cost}(Q_*^L)
  \,=\, \calO \big( \varepsilon^{-a^{\rm ML}}\, (\log\varepsilon^{-1})^{b^{\rm ML}} \big)\;,
\]
with
\begin{align*}
  a^{\rm ML}
  &\,=\,
  \begin{cases}
  \displaystyle
  \max \bigg( 2\lambda_q, \frac{d}{\tau} \bigg)
  \qquad\qquad\qquad\qquad\qquad\qquad
  \mbox{if $k$-orthogonality \eqref{eq:orthprop} holds}\;, \\
  \displaystyle
  \max \bigg(2\lambda_q,\frac{d}{\tau}\bigg)+\frac{p}{2-2p}\bigg(1-\frac{q-p}{pq}\bigg(2\lambda_q-\frac{d}{\tau}\bigg)_+\bigg)_+
  \qquad\qquad \mbox{otherwise}\;.
  \end{cases}
\end{align*}
where $\lambda_q$ is as defined in \eqref{eq:def_lam}. The value of
$b^{\rm ML}$ can be obtained from the cost bounds in Scenarios~1 and 2 in
a similar way.
\end{theorem}

In comparison, for the single level QMC FE algorithm in \cite{KSS1} to
achieve $\calO(\varepsilon)$ error, its overall cost in the case of $p<1$
is $\calO(\varepsilon^{-a^{\rm SL}})$, with
\begin{align*}
  a^{\rm SL}
  \,=\, \frac{p}{2-2p} + 2\lambda_p + \frac{d}{\tau}\;,
\end{align*}
see \cite[Theorem 8.1]{KSS1}, where $\lambda_p$ is defined analogously to
$\lambda_q$ as follows
\[
  \lambda_p \,:=\,
  \begin{cases}
  \displaystyle\frac{1}{2-2\delta} \quad\mbox{for some}\quad \delta\in (0,1/2) & \mbox{when } p\in (0,2/3]\;,
  \\
  \displaystyle\frac{p}{2-p} & \mbox{when } p\in (2/3,1)\;.
  \end{cases}
\]
Note that $a^{\rm ML}$ is much smaller than $a^{\rm SL}$ in most cases.
This is clearly seen when $\lambda_q \approx \lambda_p$. However, in the
extreme case where $\lambda_q$ and $\lambda_p$ are furthest apart, i.e,
$\lambda_q = 1$ and $\lambda_p \approx 1/2$, it is possible to come up
with an example where $a^{\rm SL} < a^{\rm ML}$: indeed, we could take
$d=1$, $\tau = 2$, $q = 1$ and $p=1/3$, which yield $a^{\rm SL} \approx
1.75$ while $a^{\rm ML} = 2$ under $k$-orthogonality. In a number of
examples it can be shown that $q = p/(1-p)$, in which case the requirement
that $q\le 1$ implies $p\le 1/2$, which is stronger than just $p\le 1$ as
required in the single level algorithm.

Now we compare with some multi-level MC and QMC works in the literature.
Sometimes ``finite-dimensional noise" is assumed, a feature we can mimic
by setting $p=q=0$ in our analysis, leading to $a^{\rm ML} = \max
(1/(1-\delta), d/\tau)$. In \cite{CST,CGST,TSGU}, multi-level MC FE
methods for elliptic PDEs \eqref{eq:PDE1} were analyzed, however with the
random coefficient \eqref{eq:defaxy} being lognormal, i.e., the
exponential of a stationary, Gaussian process.

In \cite{HMNR11} a class of abstract multi-level QMC algorithms for
infinite-dimensional integration was introduced, with a general cost model
for the evaluation of the integrand function. The multi-level structure in
that paper is different from ours: the key difference is that our
multi-level scheme must also incorporate the multi-level structure of the
FE discretizations. Also new is the necessity of considering `mixed'
regularity (in weighted reproducing kernel Hilbert spaces with respect to
the parameter sequence $\bsy$ and in the smoothness scale $Z^t$ with
respect to the spatial variable~$\bsx$).

In \cite{BSZ} a multi-level MC FE method with finite dimensional noise was
analyzed. It was shown there that in domains $D\subset \bbR^2$, a FE
approximation of the expectation of the random solution with the
convergence rate $\calO(h_L)$ {\em in the norm of $V$} (rather than for
linear functionals of the solution) can be computed in $\calO(M_{h_L}) =
\calO(h_L^{-2})$ work and memory, i.e., with the same cost as one
multi-level solution of the deterministic problem.

\section{Conclusion}
\label{sec:concl}
This paper introduces a multi-level QMC FE method, applied to functionals
of the solution of the same PDE with random coefficient problem as
considered by~\cite{CDS}. The same problem was studied by the present
authors in \cite{KSS1}, where we developed a single level QMC analysis
which yielded the same error bounds as in \cite{CDS} within the range of
convergence rates relevant to QMC. The probability model in these papers,
namely, independent and uniformly distributed parameters $y_j$, is
particularly simple and lends itself naturally to an error analysis by
QMC. The aim of the present multi-level version of the QMC approach is to
outline the design of a multilevel QMC FE Method which significantly
reduces the costs, while maintaining the fast convergence (compared to MC)
associated with QMC. We emphasize that the multi-level version requires a
new analysis, and in particular leads to a new prescription for the POD
weights (different from that in \cite{KSS1}) that determine the QMC rule.
Another difference is that the regularity requirements on the functions
$\psi_j$ are also more stringent than in the single level case.

The principal results for dimension $d=2$ are as follows. In Scenario 1
where $k$-orthogonality \eqref{eq:orthprop} holds, if we can choose
$t=t'=1$ so that $\tau = 2$, and can choose $\lambda =1/(2-2\delta)$ for
some $\delta\in (0,1/2)$, then the cost of the multi-level QMC FE
algorithm for computing the expectation of $G(u)$ is
$\calO(2^{2L/(1-\delta)}) = \calO(h_L^{2/(1-\delta)})$, while the
convergence rate is the (best possible) second order $\calO(2^{-2L}) =
\calO(h_L^2)$. This corresponds to optimal accuracy versus work bounds for
the computation of solution functionals in first order FE methods applied
to deterministic, $H^2$ regular, second order elliptic problems (see, e.g.
\cite{Ciarlet}). In contrast, multi-level MC FE methods such as those
analyzed in \cite{CST,CGST} cannot achieve optimal complexity for output
functionals for general, sufficiently regular covariances of the random
field $a(\bsx,\bsy)$, due to the maximal convergence rate $1/2$ of
standard MC methods.

As noted earlier, our cost model does not include the pre-computation cost
for the CBC construction of lattice rules. This is justified because the
same lattice rules can be used for the PDE problem with different forcing
terms $f$. However, as we are tailoring the choice of weights to the
problem, the cost of the CBC construction may be a significant issue.

The present analysis was performed under Lipschitz assumptions on $\psi_j$
and $\bar{a}$ in (\textbf{A4}) and (\textbf{A7}) which, together with
(\textbf{A6}) and the assumption that $G\in L^2(D)$, ensure in
\eqref{eq:defZ} that $Z = (H^1_0\cap H^2)(D)$ and, in turn, implies
$\calO(h^2)$ convergence in \eqref{eq:FE2}. The present convergence
analysis extends directly to weaker assumptions: if in (\textbf{A4}) and
(\textbf{A7}) we have only H\"older continuity $C^{0,r}(\overline{D})$ for
some $0<r<1$ instead of $W^{1,\infty}(D)$ regularity, or if $D$ is not
convex, then $\bar{b}_j$ in \eqref{eq:defbarbj} and
\eqref{eq:choiceweight-1} 
will depend on $\|\psi_j\|_{C^{0,r}(\overline{D})}$ rather than on
$\|\psi_j\|_{W^{1,\infty}(D)}$.

In Theorems~\ref{thm:AN} and~\ref{thm:trunc-ML-new} we considered only the
weighted Sobolev space norm involving mixed first derivatives with respect
to $\bsy$, but Theorem~\ref{thm:reg} holds for higher order mixed
derivatives. The results here can be extended by considering higher order
QMC methods, see e.g.\ \cite[Chapter~15]{DiPi10}.

Finally, in our multi-level scheme we assumed that {\em exact
expectations} $\mathbb{E}[\cdot]$ over all realizations of random shifts
$\bsDelta_\ell\in [0,1]^{s_\ell}$ are available. In practical
realizations, these expectations must be approximated by MC estimates
$E_{m_\ell}[\cdot]$ based on a finite number $m_\ell$ of i.i.d.
realizations of the shift $\bsDelta_\ell$ at discretization level $\ell =
0,1,...,L$. This leads to a further error $(\mathbb{E} -
E_{m_\ell})[\cdot]$ in term $\ell$ of \eqref{eq:T2} of order
$\calO(m_\ell^{-1})$. We can maintain our error-versus cost estimates in
\S\ref{ssec:SumCostErr}, with the same choices of parameters $s_\ell$ and
$N_\ell$, by taking $m_\ell=m^*$ independent of~$\ell$, that is, a
level-independent, fixed number of random shifts $\bsDelta_\ell$ for each
level $\ell$. To provide a reasonable error estimate, our experience
(stemming, in part, from Monte-Carlo simulations) is that the number $m^*$
of realizations of random shifts needs to be of the order of $10$ to $30$.

\begin{acknowledgements}
The authors thank Mike Giles and Robert Scheichl for valuable discussions.
Frances Kuo was supported by an Australian Research Council QEII
Fellowship, an Australian Research Council Discovery Project, and the
Vice-Chancellor's Childcare Support Fund for Women Researchers at the
University of New South Wales. Christoph Schwab was supported by the Swiss
National Science Foundation under Grant No. 200021-120290/1, and by the
European Research Council under FP7 grant AdG247277. Ian Sloan was
supported by the Australian Research Council. Part of this work was
completed during the Hausdorff Research Institute for Mathematics
Trimester Program on Analysis and Numerics for High Dimensional Problems
in 2011.
\end{acknowledgements}


\begin{thebibliography}{99}
\bibitem{BS66}
  {\sc E.~Bach and J.~Shallit},
  {\em Algorithmic Number Theory (Volume I: Efficient Algorithms)},
  MIT Press, Cambridge, 1966.
\bibitem{BSZ}
  {\sc A.~Barth, Ch.~Schwab, and N.~Zollinger},
  {\em Multi-level Monte Carlo finite element method for elliptic PDEs
  with stochastic coefficients},
  Numer.\ Math., {\bf 119} (2011), pp.~123--161.
\bibitem{CST} {\sc J.~Charrier, R.~Scheichl, and A.~L.~Teckentrup},
    {\em Finite element error analysis of elliptic PDEs with random
    coefficients and its application to multilevel Monte Carlo
    methods}, SIAM J.\ Numer.\ Anal., {\bf 51} (2013), pp.~322--352.
\bibitem{Ciarlet}
  {\sc P.~G.~Ciarlet},
  {\em The Finite Element Method for Elliptic Problems},
  Elsevier, Amsterdam 1978.
\bibitem{CGST}
  {\sc K.~A.~Cliffe, M.~B.~Giles, R.~Scheichl, and A.~L.~Teckentrup},
  {\it Multilevel Monte Carlo methods and applications to elliptic PDEs with
  random coefficients}, Computing and Visualization in Science Science 14 (2011),
  pp.~3--15.
\bibitem{CDS}
  {\sc A.~Cohen, R.~De Vore and Ch.~Schwab},
  {\em Convergence rates of best $N$-term Galerkin approximations
  for a class of elliptic sPDEs},
  Found.\ Comp.\ Math., {\bf 10} (2010), pp.~615--646.
\bibitem{CKN06}
  {\sc R.~Cools, F.~Y.~Kuo, and D.~Nuyens},
  {\em Constructing embedded lattice rules for multivariate integration},
  SIAM J.~Sci.\ Comput., {\bf 28} (2006), pp.~2162--2188.
\bibitem{DKU}
  {\sc W.~Dahmen, A.~Kunoth, and K.~Urban},
  {\em Biorthogonal spline wavelets on the interval -- stability and
              moment conditions},
  Appl. Comput. Harmon. Anal. {\bf 6} (1999), pp.~132--196.
\bibitem{D04}
  {\sc J.~Dick},
  {\em On the convergence rate of the component-by-component construction
  of good lattice rules},
  J.~Complexity, {\bf 20} (2004), pp.~493--522.
\bibitem{DKS13}
  {\sc J.~Dick, F.~Y.~Kuo, I.~H.~Sloan},
  {\em High-dimensional integration: the Quasi-Monte Carlo way},
  Acta Numer. \textbf{22} (2013), pp.~133--288.
\bibitem{DiPi10}
  {\sc J.~Dick and F.~Pillichshammer},
  {\em Digital Nets and Sequences}, Cambridge University Press, 2010.
\bibitem{DPW08}
  {\sc J.~Dick, F.~Pillichshammer, and B.~J.~Waterhouse},
  {\em The construction of good extensible rank-$1$ lattices},
  Math.\ Comp., {\bf 77} (2008), pp.~2345--2374.
\bibitem{DSWW04}
  {\sc J.~Dick, I.~H.~Sloan, X.~Wang, and H.~Wo\'zniakowski},
  {\em Liberating the weights},
  J.~Complexity, {\bf 20} (2004), pp.~593--623.
\bibitem{GilbargTrudinger}
  {\sc D.~Gilbarg and N.~S.~Trudinger},
  {\em Elliptic Partial Differential Equations of Second Order},
  Springer-Verlag, New York, 2nd Ed., 2001.
\bibitem{Gil07}
  {\sc M.B.~Giles},
  {\em Improved multilevel Monte Carlo convergence using the Milstein
  scheme},
  Monte Carlo and Quasi-Monte Carlo methods 2006, pp~343--358,
  Springer, 2007.
\bibitem{Gil08}
  {\sc M.B.~Giles},
  {\em Multilevel Monte Carlo path simulation},
  Oper.\ Res.\ {\bf 256} (2008), pp.~981--986.
\bibitem{Gne11}
  {\sc M.~Gnewuch},
  {\em Infinite-dimensional integration on weighted Hilbert spaces},
  Math.\ Comp., {\bf 81} (2012), pp.~2175--2205.
\bibitem{GKNSSS}
  {\sc I.~G.~Graham, F.~Y.~Kuo, J.~Nichols, R.~Scheichl, Ch.~Schwab, and I.~H.~Sloan},
  {\em QMC FE methods for PDEs with log-normal random coefficients}
  (in review).
\bibitem{GKNSS11}
  {\sc I.~G.~Graham, F.~Y.~Kuo, D.~Nuyens, R.~Scheichl, and I.~H.~Sloan},
  {\em Quasi-Monte Carlo methods for elliptic PDEs with random coefficients and applications},
  J.\ Comput.\ Phys., {\bf 230} (2011), pp.~3668--3694.
\bibitem{Hei01}
  {\sc S.~Heinrich},
  {\em Multilevel Monte Carlo methods},
  Lecture notes in Compu.\ Sci.\ Vol.\ 2179, pp.~3624--3651, Springer, 2001.
\bibitem{NHMR10}
  {\sc F.~J.~Hickernell, T.~M\"uller-Gronbach, B.~Niu, and K.~Ritter},
  {\em Multi-level Monte Carlo algorithms for infinite-dimensional
  integration on $\bbR^\bbN$},
  J.\ Complexity, {\bf 26} (2010), pp.~229--254.
\bibitem{K03}
  {\sc F.~Y.~Kuo},
  {\em Component-by-component constructions achieve the optimal rate
  of convergence for multivariate integration in weighted Korobov and Sobolev spaces},
  J.~Complexity, 19 (2003), pp.~301--320.
\bibitem{KSS-survey} {\sc F.~Y.~Kuo, Ch.~Schwab, and I.~H.~Sloan},
    {\em Quasi-Monte Carlo methods for high dimensional integration:
    the standard weighted-space setting and beyond}, ANZIAM J.\ {\bf 53} (2011),
         pp~1--37.
\bibitem{KSS1} {\sc F.~Y.~Kuo, Ch.~Schwab, and I.~H.~Sloan},
  \textit{Quasi-Monte Carlo finite element methods for a class of elliptic partial
  differential equations with random coefficient},
  SIAM J.\ Numer.\ Anal., {\bf 50} (2012), pp.~3351--3374.
\bibitem{KSWW10b}
  {\sc F.~Y.~Kuo, I.~H.~Sloan, G.~W.~Wasilkowski, and H.~Wo\'zniakowski},
  {\em Liberating the dimension},
  J.~Complexity, {\bf 26} (2010), pp.~422--454.
\bibitem{NguyenDiss}
  {\sc H.~Nguyen},
   {\em Finite element wavelets for solving partial differential equations},
   Ph.D. Thesis, Department of Mathematics, Universiteit Utrecht,
   The Netherlands, 2005.
\bibitem{HMNR11}
  {\sc B.~Niu, F.J.~Hickernell, T.~M\"uller-Gronbach, and K.~Ritter},
  {\em Deterministic multi-level algorithms for
  infinite-dimensional integration on $\bbR^\bbN$},
  J.\ Complexity, 27 (2011), pp.~331--351.
\bibitem{NC06a}
  {\sc D.~Nuyens and R.~Cools},
  {\em Fast algorithms for component-by-component construction of
  rank-$1$ lattice rules in shift-invariant reproducing kernel
  Hilbert spaces},
  Math.\ Comp., 75 (2006), pp.~903--920.
\bibitem{NC06b}
 {\sc D.~Nuyens and R.~Cools}, {\em Fast
    component-by-component construction of rank-$1$ lattice rules with
    a non-prime number of points}, J.\ Complexity, 22 (2006),
    pp.~4--28.
\bibitem{PW11}
  {\sc L.~Plaskota and G.~W.~Wasilkowski},
  {\em Tractability of infinite-dimensional integration in the worst
    case and randomized settings},
  J.\ Complexity, 27 (2011), pp.~505--518.
\bibitem{SchwabGittelsonActNum11}
  {\sc Ch. Schwab and C.~J. Gittelson},
  {\em Sparse tensor discretizations of high-dimensional parametric and stoch
  astic PDEs},
  Acta Numerica {\bf 20} (2011), Cambridge University Press.
\bibitem{ST06}
  {\sc Ch.~Schwab and R.~A.~Todor},
  {\em \KL approximation of random fields by generalized fast multipole methods},
  J.\ Comput.\ Phy., 217 (2006), pp.~100--122.
\bibitem{SKJ02b}
  {\sc I.~H.~Sloan, F.~Y.~Kuo, and S.~Joe},
  {\em Constructing randomly shifted lattice rules in weighted Sobolev
  spaces},
  SIAM J.~Numer.\ Anal., 40 (2002), pp.~1650--1665.
\bibitem{SW98}
  {\sc I.~H.~Sloan and H.~Wo\'zniakowski},
  {\em When are quasi-Monte Carlo algorithms efficient for
  high-dimensional integrals?},
  J.~Complexity, 14 (1998), pp.~1--33.
\bibitem{SWW04}
  {\sc I.~H.~Sloan, X.~Wang, and H.~Wo\'zniakowski},
  {\em Finite-order weights imply tractability of multivariate integration},
  J.~Complexity, 20 (2004), pp.~46--74.
\bibitem{TSGU}
  {\sc A.~L.~Teckentrup, R.~Scheichl, M.~B.~Giles, and E.~Ullmann}
  {\em Further analysis of multilevel Monte Carlo methods for elliptic
  PDEs with random coefficient},
  Numer.\ Math., 125 (2013), pp.~569--600.
\end{thebibliography}
\end{document}